\documentclass[final]{siamltex}
\usepackage{epsfig}
\usepackage{color}
\usepackage{amsmath}
\usepackage{amscd}
\usepackage{amsmath, amsfonts, amssymb,mathrsfs}
\usepackage{multirow}
\usepackage{enumerate}
\usepackage{verbatim}
\usepackage{cite}
\usepackage{tikz}
\usetikzlibrary{matrix,arrows}
\usepackage[justification=raggedright]{caption}
\usepackage{subcaption}
\usepackage[ruled,procnumbered]{algorithm2e}
\usepackage{graphicx}
\usepackage{booktabs}

\numberwithin{equation}{section}

\makeatletter
\def\hlinewd#1{%
\noalign{\ifnum0=`}\fi\hrule \@height #1 %
\futurelet\reserved@a\@xhline}
\makeatother

\def \vec#1{{\bf{#1}}}
\def\pd#1#2{\frac{\partial #1}{\partial #2}}

\newcommand{\bi}{\begin{itemize}}
\newcommand{\ei}{\end{itemize}}

\newcommand{\diverg}{\vec{\nabla}\cdot}
\newcommand{\director}{\vec{n}}
\newcommand{\curl}{\vec{\nabla}\times}
\newcommand{\Ltwoinner}[3]{\langle #1,#2 \rangle_0}
\newcommand{\Ltwonorm}[2]{\Vert #1 \Vert_0}
\newcommand{\Ltwonormndim}[3]{\Vert #1 \Vert_0}
\newcommand{\Ltwoinnerndim}[4]{\langle #1,#2 \rangle_0}

\newcommand{\diff}[1]{\, d#1}

\newcommand{\ltwonorm}[1]{\vert #1 \vert}
\newcommand{\ltwoinner}[2]{( #1, #2 )}
\newcommand{\kdirector}{\vec{n}_k}
\newcommand{\ddirector}{\delta \director}
\newcommand{\dlambda}{\delta \lambda}
\newcommand{\klambda}{\lambda_k}
\DeclareMathOperator*{\argmin}{argmin}

\newcommand{\lagdivn}{\mathcal{L}_{\director}[\vec{v}]}
\newcommand{\lagdivlam}{\mathcal{L}_{\lambda}[\gamma]}

\newcommand{\Hdc}{\mathcal{H}^{DC}{(\Omega)}}
\newcommand{\Hdcnot}{\mathcal{H}^{DC}_0{(\Omega)}}

\newcommand{\Honenot}[1]{H^1_0({#1})}

\newcommand{\Ltwo}[1]{L^2(#1)}

\newcommand{\Lp}[1]{L^p (\Omega)}
\newcommand{\Hcurl}[1]{H(\text{curl},#1)}
\newcommand{\Hdiv}[1]{H(\text{div},#1)}

\newcommand{\Linfinity}[1]{L^{\infty}(\Omega)}

\newcommand{\Hdcnorm}[2]{\Vert #1 \Vert_{DC}}

\newcounter{casenum}
\newenvironment{caseof}{\setcounter{casenum}{1}}{}
\newcommand{\case}[2]{\vskip.5\baselineskip\par\noindent {\bfseries Case \arabic{casenum}.} #1 \\ #2\addtocounter{casenum}{1}}

\newtheorem{assumption}[theorem]{Assumption}

\title{Constrained Optimization for Liquid Crystal Equilibria: Extended Results}

\author{J. H. Adler, D. B. Emerson, S. P. MacLachlan, T. A. Manteuffel}

\begin{document}

\maketitle

\begin{abstract}
This paper investigates energy-minimization finite-element approaches for the computation of nematic liquid crystal equilibrium configurations. We compare the performance of these methods when the necessary unit-length constraint is enforced by either continuous Lagrange multipliers or a penalty functional. Building on previous work in \cite{Emerson1, Emerson2}, the penalty method is derived and the linearizations within the nonlinear iteration are shown to be well-posed under certain assumptions. In addition, the paper discusses the effects of tailored trust-region methods and nested iteration for both formulations. Such methods are aimed at increasing the efficiency and robustness of each algorithms' nonlinear iterations. Three representative, free-elastic, equilibrium problems are considered to examine each method's performance. The first two configurations have analytical solutions and, therefore, convergence to the true solution is considered. The third problem considers more complicated boundary conditions, relevant in ongoing research, simulating surface nano-patterning. A multigrid approach is introduced and tested for a flexoelectrically coupled model to establish scalability for highly complicated applications. The Lagrange multiplier method is found to outperform the penalty method in a number of measures, trust regions are shown to improve robustness, and nested iteration proves highly effective at reducing computational costs.
 \end{abstract}
 

 
 \pagestyle{myheadings}
\thispagestyle{plain}
\markboth{\sc Adler, Emerson, Maclachlan, Manteuffel}{\sc Nematic Liquid Crystal Constrained Optimization}

\section{Introduction}

Liquid crystals are substances that possess mesophases with characteristics spanning those of isotropic liquids and solid crystals. That is, liquid crystals are fluid, yet exhibit long-range structured ordering. This paper considers nematic liquid crystals which consist of rod-like molecules whose average pointwise orientation is represented by a vector, $\director(x,y,z) = (n_1, n_2, n_3)^T$. This orientation vector is known as the director and is assumed to be headless for nematics. Therefore, $\director$ and $-\director$ are indistinguishable at any point in the domain $\Omega$, due to molecular symmetry. An important constraint on the director vector field is that $\director$ remain of unit length pointwise throughout $\Omega$. Thorough overviews of liquid crystal physics are found in \cite{Stewart1, deGennes1, Chandrasekhar1}.

The deformable ordering of liquid crystal structures, coupled with the materials' birefringent and dielectric properties, has led to many important applications and discoveries, most famously in display technologies. Additional modern applications include nanoparticle organization \cite{Wan1}, photorefractive cells \cite{Daly1}, and liquid crystal elastomers designed to produce effective actuator devices such as light driven motors \cite{Yamada1} and artificial muscles \cite{Thomsen1}. Numerical simulations of liquid crystal equilibrium configurations are used to optimize the performance of devices, analyze experiments, validate theory, and suggest the presence of new physical phenomenon \cite{Atherton1, Emerson2}. Many current technologies and experiments, including bistable devices \cite{Davidson1, Majumdar1}, require simulations with anisotropic physical constants on two- and three-dimensional domains.

To this end, a theoretically supported energy-minimization finite-element approach using Newton linearization and a Lagrange multiplier for the pointwise constraint was developed in \cite{Emerson1, Emerson2}. The approach effectively enforces the unit-length constraint while converging to energy-minimizing configurations. However, alternative approaches to efficiently impose unit-length conformance exist, such as the renormalized Newton method presented in \cite{Ramage3}. Penalty methods have also been applied to liquid crystal equilibrium problems \cite{Atherton1, Glowinski1, Hu1} and are utilized extensively to simplify the Leslie-Ericksen equations \cite{Ericksen1, Leslie1} in nematohydrodynamics simulations \cite{Liu2, Liu4, Liu5}. In addition, penalty methods are used for unit-length constraints in certain ferromagnetic problems \cite{Kruzik1}. 

In this paper, we focus on the Frank-Oseen free-elastic energy model and aim to compare the performance of techniques enforcing the unit-length constraint via Lagrange multipliers or with penalty methods employing augmentations to the free-elastic energy functional. The energy model and approaches are discussed in Section \ref{energymodels}. In Section \ref{penaltywellposedness}, well-posedness for the intermediate Newton linearizations that arise in the penalty method formulation is established. In addition to the incomplete Newton stepping discussed in \cite{Emerson1}, several tailored trust-region methods are investigated in Section \ref{TrustRegionMethods}. These trust-region approaches include one- and two-dimensional subspace minimization techniques \cite{Byrd1, Byrd2, Nocedal1}. A modified penalty method, which normalizes the director after each step, is also introduced in this section. The resulting algorithms are tested on three benchmark free-elastic problems in Section \ref{nummethodology}. Two of the problems have analytical solutions and one simulates nano-patterned boundary conditions relevant in ongoing research \cite{Atherton1, Atherton2}. In each of the experiments conducted, the Lagrange multiplier method outperforms the penalty approaches in a number of measures. Moreover, trust-regions are shown to improve convergence robustness and nested iteration proves exceptionally effective at reducing overall computational costs. In Section \ref{BraessSarazinSmoothing}, a Braess-Sarazin-type multigrid scheme \cite{Braess2, Benson1} is introduced for a flexoelectrically coupled problem with nano-patterned boundary conditions, which demonstrates scalability for highly complicated models with coupled physics. Finally, Section \ref{conclusion} gives some concluding remarks, and future work is discussed.

\section{Energy Model and Minimization} \label{energymodels}

While a number of free-elastic energy models exist \cite{Davis1, Frank1, Stewart1}, we consider the Frank-Oseen free-elastic energy herein \cite{Virga1, Stewart1}. The equilibrium, free-elastic energy is represented by an integral functional depending on deformations of the director field, $\director$. Liquid crystal samples favor stable configurations attaining minimal free energy. As in \cite{Emerson1, Emerson2}, let $K_i$, $i=1,2,3$, be the Frank constants \cite{Frank1} with $K_i \geq 0$ by Ericksen's inequalities \cite{Ericksen2}, and let
\begin{equation*} \label{matrixD}
\vec{Z} = \kappa \director \otimes \director + (\vec{I} - \director \otimes \director) = \vec{I} - (1-\kappa) \director \otimes \director,
\end{equation*}
where $\kappa = K_2/K_3$. In general, we consider the case that  $K_2, K_3 \neq 0$. Denote the classical $\Ltwo{\Omega}$ inner product and norm as $\Ltwoinner{\cdot}{\cdot}{\Omega}$ and $\Ltwonorm{\cdot}{\Omega}$, respectively, and the standard Euclidean inner product and norm as $(\cdot, \cdot)$ and $\vert \cdot \vert$. Throughout this paper, we assume the presence of Dirichlet boundary conditions or mixed Dirichlet and periodic boundary on a rectangular domain and, therefore, utilize the null Lagrangian simplification discussed in \cite{Emerson1, Emerson2, Stewart1}. Thus, the Frank-Oseen free-elastic energy for a domain, $\Omega$, is written
\begin{align} \label{FrankOseenFree}
\int_{\Omega} w_F \diff{V} =& \frac{1}{2}K_1 \Ltwonorm{\diverg \director}{\Omega}^2 +\frac{1}{2}K_3\Ltwoinnerndim{\vec{Z} \curl \director}{\curl \director}{\Omega}{3}.
\end{align}
Note that if $\kappa =1$, $\vec{Z}$ is reduced to the identity and the energy becomes a Div-Curl system. The Frank constants may be determined experimentally for different liquid crystal types, are often anisotropic (i.e. $K_1 \neq K_2 \neq K_3$), and may depend on temperature \cite{deGennes1, Haller1}.

The one-constant approximation such that $K_1 = K_2 = K_3$, discussed in \cite{Emerson1, Emerson2, Stewart1}, is a widely applied simplification \cite{Ramage1, Liu1, Liu4, Stewart1, Cohen1} of \eqref{FrankOseenFree}. This simplification significantly reduces the complexity of the free-elastic energy functional for both theoretical analysis and computational simulation. However, it ignores anisotropic physical characteristics, which play major roles in liquid crystal phenomena \cite{Lee1, Atherton2}. Therefore, both the penalty and Lagrange multiplier methods considered here do not rely on such an assumption.

\subsection{Penalty and Lagrange Multiplier Energy Minimization} \label{MethodsDiscussion}

The admissible equilibrium state for a liquid crystal sample is one that minimizes the system free energy in \eqref{FrankOseenFree}, subject to the local constraint $\ltwoinner{\director}{\director} = 1$. Here, we discuss two energy-minimization approaches, imposing this constraint via a penalty method or with Lagrange multipliers. To compute the free-energy minimizing configurations, we define the functional
\begin{align} \label{functional2}
\mathcal{F}(\director) &= K_1 \Ltwonorm{\diverg \director}{\Omega}^2 + K_3\Ltwoinnerndim{\vec{Z} \curl \director}{\curl \director}{\Omega}{3}.
\end{align}
Throughout this paper, we will make use of the spaces 
\begin{align*}
\Hdiv{\Omega} &= \{\vec{v} \in L^2(\Omega)^3 : \diverg \vec{v} \in L^2(\Omega) \},\\
\Hcurl{\Omega} &= \{ \vec{v} \in L^2(\Omega)^3 : \curl \vec{v} \in L^2(\Omega)^3 \},
\end{align*}
as well as
\begin{equation*}
\Hdc= \{ \vec{v} \in \Hdiv{\Omega} \cap \Hcurl{\Omega} : B(\vec{v}) = \vec{g} \},
\end{equation*}
with norm $\Hdcnorm{\vec{v}}{\Omega}^2 = \Ltwonormndim{\vec{v}}{\Omega}{3}^2 + \Ltwonorm{\diverg \vec{v}}{\Omega}^2 + \Ltwonormndim{\curl \vec{v}}{\Omega}{3}^2$ and appropriate boundary conditions $B(\vec{v})=\vec{g}$. Here, we assume that $\vec{g}$ satisfies appropriate compatibility conditions for the operator $B$. For example, if $B$ represents full Dirichlet boundary conditions and $\Omega$ has a Lipschitz continuous boundary, it is assumed that $\vec{g} \in H^{\frac{1}{2}}(\partial \Omega)^3$ \cite{Girault1}. Furthermore, let $\Hdcnot = \{ \vec{v} \in \Hdiv{\Omega} \cap \Hcurl{\Omega} : B(\vec{v}) = \vec{0} \}$. Note that if $\Omega$ is a Lipshitz domain and $B$ imposes full Dirichlet boundary conditions on all components of $\vec{v}$, then $\Hdcnot = \Honenot{\Omega}^3$ \cite[Lemma 2.5]{Girault1}.

\subsection{Lagrange Multiplier Formulation} \label{LagrangeMultiplierApproach}

For this approach, the pointwise unit-length constraint is imposed by a continuous Lagrange multiplier. Following \cite{Emerson1, Emerson2}, the Lagrangian is defined as
\begin{align*}
\mathcal{L}(\director, \lambda) &= \mathcal{F}(\director) + \int_{\Omega} \lambda(\vec{x})(\ltwoinner{\director}{\director}-1) \diff{V},
\end{align*}
where $\lambda \in \Ltwo{\Omega}$. First-order optimality conditions, given by
\begin{align*}
\lagdivn &= \frac{\partial}{\partial \director} \mathcal{L}(\director, \lambda) [\vec{v}] =0, & & \forall \vec{v} \in \Hdcnot, \\
\lagdivlam &= \frac{\partial}{\partial \lambda} \mathcal{L}(\director, \lambda) [\gamma] =0,& & \forall \gamma \in L^2(\Omega),
\end{align*}
are derived and linearized to yield the Newton update equations
\begin{equation} \label{newtonhessian}
\left [ \begin{array}{c c}
\mathcal{L}_{\director \director} &  \mathcal{L}_{\director \lambda} \\
\mathcal{L}_{\lambda \director} & \vec{0}
\end{array} \right ]
\left [ \begin{array}{c} 
\ddirector \\
\dlambda
\end{array} \right]
= - \left[ \begin{array}{c}
\mathcal{L}_{\director} \\
\mathcal{L}_{\lambda} 
\end{array} \right],
\end{equation}
where each of the system components are evaluated at the current approximations $\kdirector$ and $\klambda$, while $\ddirector= \director_{k+1} - \kdirector$ and $\dlambda = \lambda_{k+1}-\klambda$ are the desired updates to these approximations. The matrix-vector multiplication indicates the direction that the derivatives in the Hessian are taken. That is,
\begin{align*}
\mathcal{L}_{\director \director}[\vec{v}] \cdot \ddirector = \pd{ }{\director} \left ( \mathcal{L}_{\director} (\kdirector, \klambda)[\vec{v}]\right)[\ddirector], & &
\mathcal{L}_{\director \lambda}[\vec{v}] \cdot \dlambda = \pd{ }{\lambda} \left( \mathcal{L}_{\director} (\kdirector, \klambda)[\vec{v}] \right ) [\dlambda], \\
\mathcal{L}_{\lambda \director}[\gamma] \cdot \ddirector = \pd{ }{\director} \left( \mathcal{L}_{\lambda} (\kdirector, \klambda)[\gamma] \right)[\ddirector], & &
\end{align*}
where the partials denote G\^{a}teaux derivatives in the respective variables. The above system represents a linearized variational system for which we seek solutions $\ddirector$ and $\dlambda$. The complete system is found in \cite{Emerson1}.

\subsection{Penalty Method Formulation}

In order to define the penalty approach, the free-energy functional in \eqref{functional2} is augmented with a weighted, positive term,
\begin{align}
\mathcal{P}(\director) &= K_1\Ltwoinner{\diverg \director}{\diverg \director}{\Omega} + K_3\Ltwoinnerndim{\vec{Z} \curl \director}{\curl \director}{\Omega}{3} + \zeta \Ltwoinner{\director \cdot \director -1}{\director \cdot \director -1}{\Omega} \label{Dirichletpenaltyfunctional},
\end{align}
where $\zeta > 0$ represents a constant weight, penalizing deviations of the solution from the unit-length constraint. Thus, in the limit of large $\zeta$ values, unconstrained minimization of \eqref{Dirichletpenaltyfunctional} is equivalent to the constrained minimization of \eqref{functional2}. In order to minimize $\mathcal{P}(\director)$, we compute the G\^{a}teaux derivative of $\mathcal{P}(\director)$ with respect to $\director$ in the direction $\vec{v} \in \Hdcnot$. Hence, the first-order optimality condition is
\begin{align*}
\mathcal{P}_{\director}[\vec{v}] = \pd{}{\director}\mathcal{P}(\director)[\vec{v}] = 0, & & \forall \vec{v} \in \Hdcnot.
\end{align*}
Computation of this derivative yields the variational problem
\begin{align*}
\mathcal{P}_{\director}[\vec{v}] &= 2K_1\Ltwoinner{\diverg \director}{\diverg \vec{v}}{\Omega} + 2K_3\Ltwoinnerndim{\vec{Z} \curl \director}{\curl \vec{v}}{\Omega}{3} \nonumber  \\
& \qquad+ 2(K_2-K_3)\Ltwoinner{\director \cdot \curl \director}{\vec{v} \cdot \curl \director}{\Omega} + 4 \zeta \Ltwoinner{\vec{v} \cdot \director}{\director \cdot \director -1}{\Omega} = 0,
\end{align*}
for all $\vec{v} \in \Hdcnot$.

As with the Lagrangian formulation, the variational problem above contains nonlinearities. Therefore, Newton iterations are again applied, requiring computation of the second-order G\^{a}teaux derivative with respect to $\director$. Let $\kdirector$ be the current approximation for $\director$ and $\ddirector= \director_{k+1} - \kdirector$ be the update that we seek to compute. Then, the Newton linearizations are written
\begin{align} \label{penaltylinearization}
\pd{}{\director} \left( \mathcal{P}_{\director}(\kdirector)[\vec{v}] \right)[\ddirector] = - \mathcal{P}_{\director}(\kdirector)[\vec{v}], & & \forall \vec{v} \in \Hdcnot, 
\end{align}
where
\begin{align}
\pd{}{\director} \left( \mathcal{P}_{\director}(\kdirector)[\vec{v}] \right)[\ddirector]  &= 2K_1\Ltwoinner{\diverg \ddirector}{\diverg \vec{v}}{\Omega} + 2K_3 \Ltwoinnerndim{\vec{Z}(\kdirector) \curl \ddirector}{\curl \vec{v}}{\Omega}{3} \nonumber \\
& \qquad + 2(K_2-K_3) \Big(\Ltwoinner{\ddirector \cdot \curl \vec{v}}{\kdirector \cdot \curl \kdirector}{\Omega} \nonumber \\
& \qquad +\Ltwoinner{\kdirector \cdot \curl \vec{v}}{\ddirector \cdot \curl \kdirector}{\Omega} + \Ltwoinner{\kdirector \cdot \curl \kdirector}{\vec{v} \cdot \curl \ddirector}{\Omega} \nonumber \\
& \qquad + \Ltwoinner{\kdirector \cdot \curl \ddirector}{\vec{v} \cdot \curl \kdirector}{\Omega} + \Ltwoinner{\ddirector \cdot \curl \kdirector}{\vec{v} \cdot \curl \kdirector}{\Omega}\Big) \nonumber \\
& \qquad + 4\zeta \big( \Ltwoinner{\kdirector \cdot \kdirector -1}{\vec{v} \cdot \ddirector}{\Omega} + 2 \Ltwoinner{\ddirector \cdot \kdirector}{\vec{v} \cdot \kdirector}{\Omega} \big). \label{penaltysecondderivative}
\end{align}
Completing \eqref{penaltylinearization} with the above second-order derivative computation yields a linearized variational system. For each iteration, we compute $\ddirector$ satisfying \eqref{penaltylinearization} for all $\vec{v} \in \Hdcnot$ with the current approximation $\kdirector$.

\section{Well-Posedness of the Penalty Newton Linearizations} \label{penaltywellposedness}

Existence and uniqueness of solutions to the discrete form of the linearization systems described in \eqref{newtonhessian} is established in \cite{Emerson1} under reasonable assumptions. Here, we adapt these results to the discrete form of the penalty method linearizations in Equation \eqref{penaltylinearization}. 

Let $a(\ddirector, \vec{v})$ denote the bilinear form defined in \eqref{penaltysecondderivative} for fixed $\kdirector$ and $F(\vec{v})$ be the linear functional on the right-hand side of the linearization in \eqref{penaltylinearization}. Using finite elements to approximate the desired update, $\ddirector$, and considering a discrete space $V_h \subset \Hdcnot$ yields the discrete linearized system,
\begin{align}
a(\ddirector_h, \vec{v}_h) = F(\vec{v}_h), \qquad \forall \vec{v}_h \in V_h. \label{auvform}
\end{align}

Throughout the rest of this section, the developed theory applies exclusively to discrete spaces. Therefore, except when necessary for clarity, we drop the subscript $h$ along with the notation, $\ddirector$. For instance, we write $a(\vec{u}, \vec{v})$ to indicate the bilinear form in \eqref{auvform} operating on the discrete space $V_h \times V_h$. Furthermore, we refer to the following set of assumptions. 
\begin{assumption} \label{secass2}
Consider an open bounded domain, $\Omega$, with a Lipschitz-continuous boundary. Further, assume that there exist constants $0 < \alpha \leq 1 \leq \beta$, such that $\alpha \leq \ltwonorm{\kdirector}^2 \leq \beta$ and $\vec{Z}(\kdirector(\vec{x}))$ remains uniformly symmetric positive definite (USPD) with lower and upper bounds on its Rayleigh quotient, $\eta$ and $\Lambda$, respectively, as in \cite[Lemma 2.1]{Emerson1}. Finally, assume that Dirichlet boundary conditions are applied.
\end{assumption}\\
While the assumption above, and the theory below, explicitly concern full Dirichlet boundary conditions, the theory is equally applicable to mixed Dirichlet and periodic boundary conditions on a rectangular domain.

In order to establish well-posedness of \eqref{auvform}, we show that the functional, $F(\vec{v})$, is continuous and that the bilinear form, $a(\vec{u}, \vec{v})$, is continuous and coercive. Decomposing the bilinear form, $a(\vec{u},\vec{v})$, and the linear form, $F(\vec{v})$, into terms that contain the penalty term and those that do not,
\begin{align*}
a(\vec{u},\vec{v}) = & ~\hat{a}(\vec{u},\vec{v}) + 2\zeta \big( \Ltwoinner{\kdirector \cdot
  \kdirector -1}{\vec{v} \cdot \vec{u}}{\Omega} + 2
\Ltwoinner{\vec{u} \cdot \kdirector}{\vec{v} \cdot
  \kdirector}{\Omega} \big)\\
F(\vec{v}) = & ~\hat{F}(\vec{v}) + 2 \zeta \Ltwoinner{\vec{v} \cdot
  \kdirector}{\kdirector \cdot \kdirector -1}{\Omega},
\end{align*} 
we then extend the results in \cite{Emerson1}.
\begin{lemma} \label{boundedlinearform}
Under Assumption \ref{secass2}, $F$ is a bounded linear functional on $V_h$.
\end{lemma}
\begin{proof}
From the bounds derived in \cite[Lemma 3.6]{Emerson1} and an application of the Cauchy-Schwarz inequality,
\begin{align*}
\vert F(\vec{v}) \vert & \leq \vert \hat{F}(\vec{v}) \vert + 2 \zeta
\Ltwonorm{\kdirector \cdot \kdirector -1}{\Omega}\Ltwonorm{\kdirector
  \cdot \vec{v}}{\Omega} \\
  &\leq C_F \Hdcnorm{\vec{v}}{\Omega} + 2 \zeta \Ltwonorm{\kdirector \cdot \kdirector -1}{\Omega}\Ltwonorm{\kdirector \cdot \vec{v}}{\Omega},
\end{align*}
where $C_F$ is a constant independent of mesh size, defined in \cite{Emerson1}. Note that by assumption $\alpha \leq \kdirector \cdot \kdirector \leq \beta$, where $0 < \alpha \leq 1 \leq \beta$. Then, letting $C_{\mu} = \max (1-\alpha, \beta -1)$,
\begin{align*}
\Ltwonorm{\kdirector \cdot \kdirector -1}{\Omega}^2 &= \int_{\Omega} (\kdirector \cdot \kdirector -1)^2 \diff{V} \leq C_{\mu}^2 \int_{\Omega} \diff{V} = C_{\mu}^2 \vert \Omega \vert.
\end{align*}
Hence, $\Ltwonorm{\kdirector \cdot \kdirector -1}{\Omega} \leq C_{\mu} \vert \Omega \vert^{\frac{1}{2}}$. In addition,
\begin{align*}
\Ltwonorm{\kdirector \cdot \vec{v}}{\Omega} &\leq \sqrt{\beta} \Ltwonormndim{\vec{v}}{\Omega}{3} \leq \sqrt{\beta} \Hdcnorm{\vec{v}}{\Omega}.
\end{align*}
Thus,
\begin{align*}
\vert F(\vec{v}) \vert &\leq C_F \Hdcnorm{\vec{v}}{\Omega} + 2 \zeta C_{\mu} \vert \Omega \vert^{\frac{1}{2}} \sqrt{\beta} \Hdcnorm{\vec{v}}{\Omega}.
\end{align*}
\end{proof}
\begin{lemma} \label{boundedbilinearform}
Under Assumption \ref{secass2}, $a(\vec{u}, \vec{v})$ is continuous.
\end{lemma}
\begin{proof}
Using the bounds derived in \cite[Lemma 3.7]{Emerson1},
\begin{align*}
a(\vec{u},\vec{v}) &\leq \hat{a}(\vec{u},\vec{v}) + 2\zeta \big( \vert
\Ltwoinner{\kdirector \cdot \kdirector - 1}{\vec{u} \cdot
  \vec{v}}{\Omega} \vert + 2 \vert \Ltwoinner{\vec{u} \cdot
  \kdirector}{\vec{v} \cdot \kdirector}{\Omega} \vert \big)\\
&\leq C_A \Hdcnorm{\vec{u}}{\Omega} \Hdcnorm{\vec{v}}{\Omega} + 2\zeta \big( \vert \Ltwoinner{\kdirector \cdot \kdirector - 1}{\vec{u} \cdot \vec{v}}{\Omega} \vert + 2 \vert \Ltwoinner{\vec{u} \cdot \kdirector}{\vec{v} \cdot \kdirector}{\Omega} \vert \big),
\end{align*}
where $C_A$ is the continuity constant defined in \cite{Emerson1}. Note that,
\begin{align*}
\vert \Ltwoinner{\kdirector \cdot \kdirector - 1}{\vec{u} \cdot \vec{v}}{\Omega} \vert &= \vert \Ltwoinnerndim{(\kdirector \cdot \kdirector - 1)\vec{u}}{\vec{v}}{\Omega}{3} \vert \leq \Ltwonormndim{(\kdirector \cdot \kdirector -1) \vec{u}}{\Omega}{3} \Ltwonormndim{\vec{v}}{\Omega}{3}.
\end{align*}
Furthermore,
\begin{align*}
\Ltwonormndim{(\kdirector \cdot \kdirector -1) \vec{u}}{\Omega}{3}^2 &= \int_{\Omega} (\kdirector \cdot \kdirector -1)^2 (\vec{u} \cdot \vec{u}) \diff{V} \leq C_{\mu}^2 \Ltwonormndim{\vec{u}}{\Omega}{3}^2.
\end{align*}
This implies that
\begin{equation*}
\Ltwonormndim{(\kdirector \cdot \kdirector -1) \vec{u}}{\Omega}{3} \leq C_{\mu} \Ltwonormndim{\vec{u}}{\Omega}{3},
\end{equation*}
and 
\begin{equation*}
\vert \Ltwoinner{\kdirector \cdot \kdirector - 1}{\vec{u} \cdot \vec{v}}{\Omega} \vert \leq C_{\mu} \Hdcnorm{\vec{u}}{\Omega} \Hdcnorm{\vec{v}}{\Omega}.
\end{equation*}
Noting that
\begin{equation*}
\vert \Ltwoinner{\vec{u} \cdot \kdirector}{\vec{v} \cdot \kdirector}{\Omega} \vert \leq \beta \Hdcnorm{\vec{u}}{\Omega} \Hdcnorm{\vec{v}}{\Omega},
\end{equation*}
we bound
\begin{align*}
a(\vec{u},\vec{v}) &\leq \Big(C_A + 2 \zeta \big(C_{\mu} + 2\beta \big) \Big) \Hdcnorm{\vec{u}}{\Omega} \Hdcnorm{\vec{v}}{\Omega}.
\end{align*}
\end{proof}

\noindent Following the theory established in \cite{Emerson1}, two coercivity lemmas for $a(\vec{u}, \vec{v})$ are proved. The first proof addresses the case when $\kappa = 1$. The second considers coercivity when $\kappa$ lies in a neighborhood of unity, $\kappa \in (1 - \epsilon_2, 1 + \epsilon_1)$. Let $\alpha_0 > 0$ be the coercivity constant from \cite[Lemma 3.7]{Emerson1}.

\begin{lemma} \label{coercivityauv}
Under Assumption \ref{secass2},  if $\kappa =1$ and $2 \zeta \vert \alpha -1 \vert < \alpha_0$, there exists a $\beta_0 > 0$ such that $\beta_0 \Hdcnorm{\vec{v}}{\Omega}^2 \leq a(\vec{v}, \vec{v})$ for all $\vec{v} \in V_h$.
\end{lemma}
\begin{proof}
\begin{align*}
a(\vec{v},\vec{v}) = \hat{a}(\vec{v},\vec{v}) +  2\zeta \Ltwoinner{\kdirector \cdot \kdirector -1}{\vec{v} \cdot \vec{v}}{3} + 4 \zeta \Ltwoinner{\vec{v} \cdot \kdirector}{\vec{v} \cdot \kdirector}{\Omega}
\end{align*}
Using the coercivity of $\hat{a}(\vec{v},\vec{v})$ from \cite[Lemma 3.7]{Emerson1} and the fact that $\Ltwoinner{\vec{v} \cdot \kdirector}{\vec{v} \cdot
  \kdirector}{\Omega} \geq 0$, 
\begin{equation*}
\alpha_0 \Hdcnorm{\vec{v}}{\Omega}^2 \leq \hat{a}(\vec{v},\vec{v}) + 4\zeta \Ltwoinner{\vec{v} \cdot \kdirector}{\vec{v} \cdot \kdirector}{\Omega}.
\end{equation*}
Observe that
\begin{equation*}
\Ltwoinner{\kdirector \cdot \kdirector -1}{\vec{v} \cdot \vec{v}}{\Omega} = \int_{\Omega} (\kdirector \cdot \kdirector -1)(\vec{v} \cdot \vec{v}) \diff{V}.
\end{equation*}
If $\alpha \leq \kdirector \cdot \kdirector \leq \beta$ for all $\vec{x} \in \Omega$ with $0 < \alpha \leq 1 \leq \beta$, then $(\alpha -1) \leq 0$, and
\begin{align}
\Ltwoinner{\kdirector \cdot \kdirector -1}{\vec{v} \cdot \vec{v}}{\Omega} &\geq (\alpha -1) \int_{\Omega} \vec{v} \cdot \vec{v} \diff{V} \geq (\alpha-1) \Hdcnorm{\vec{v}}{\Omega}^2 \label{penaltyboundbelow}.
\end{align}
Letting $\beta_0 = \alpha_0 - 2 \zeta \vert \alpha -1 \vert$,
\begin{equation*}
 \beta_0 \Hdcnorm{\vec{v}}{\Omega}^2 \leq \hat{a}(\vec{v},\vec{v}) + 4\zeta \Ltwoinner{\vec{v} \cdot \kdirector}{\vec{v} \cdot \kdirector}{\Omega} + 2 \zeta \Ltwoinner{\kdirector \cdot \kdirector -1}{\vec{v} \cdot \vec{v}}{\Omega}.
\end{equation*}
Hence, if $2 \zeta \vert \alpha -1 \vert < \alpha_0$, then $\beta_0 > 0$.
\end{proof}

\noindent Therefore, $a(\vec{u}, \vec{v})$ is coercive for $\kappa = 1$ if $\zeta$ is not so large in comparison to the pointwise lower bound on the director length as to overwhelm $\alpha_0$. 

As in \cite{Emerson1}, the assumption that $\kappa = 1$ can be loosened to include some anisotropy and retain coercivity of $a(\vec{u}, \vec{v})$. Let $C > 0$ such that $\Ltwonorm{\vec{v}}{\Omega}^2 \leq C \big( \Ltwonorm{\diverg \vec{v}}{\Omega}^2 + \Ltwonorm{\curl \vec{v}}{\Omega}^2 \big)$ (see \cite{Emerson1}). Further, let $\alpha_1 > 0$ be defined as in the proof of \cite[Lemma 3.8]{Emerson1}, where $K' = \min(K_1, \eta K_3)$ and $\alpha_1 = \frac{K'}{(C+1)}$. The following extends the results of \cite[Lemma 3.8]{Emerson1} to the penalty method.
\begin{lemma}[Small Data] \label{coercivitysmalldata}
Under Assumption \ref{secass2}, if 
\begin{equation*}
\beta_1 = \frac{\min(K_1, K_3)}{C+1} - 2 \zeta \vert \alpha -1 \vert > 0,
\end{equation*}
there exists $\epsilon_1, \epsilon_2 > 0$, dependent on $\beta=\max \ltwonorm{\kdirector}^2$, such that for $\kappa \in (1-\epsilon_2, 1+\epsilon_1)$, $a(\vec{u},\vec{v})$ is coercive. 
\end{lemma}
\begin{proof}
Let 
\begin{align*}
\tilde{a}(\vec{v}, \vec{v}) &= K_1 \Ltwoinner{\diverg \vec{v}}{\diverg \vec{v}}{\Omega} + K_3 \Ltwoinnerndim{\vec{Z}(\kdirector) \curl \vec{v}}{\curl \vec{v}}{\Omega}{3} + 4 \zeta \Ltwoinner{\vec{v} \cdot \kdirector}{\vec{v} \cdot \kdirector}{\Omega} \nonumber \\
& \qquad +2 \zeta \Ltwoinner{\kdirector \cdot \kdirector -1}{\vec{v} \cdot \vec{v}}{\Omega}.
\end{align*}
From the proof of \cite[Lemma 3.8]{Emerson1}, the fact that $\zeta > 0$, and \eqref{penaltyboundbelow},
\begin{equation} \label{partialcoercivitysmalldata}
(\alpha_1- 2\zeta \vert \alpha -1 \vert) \Hdcnorm{\vec{v}}{\Omega}^2 \leq \tilde{a}(\vec{v}, \vec{v}).
\end{equation}
The USPD lower bound for $\vec{Z}(\kdirector)$, $\eta$, may depend on $\kappa$ \cite[Lemma 2.1]{Emerson1}. Thus, the proof is split into three cases.
\begin{caseof}
\case{$\kappa=1+\epsilon_1$, for $\epsilon_1 >0$.}{If this case holds, then $\eta=1$. Hence, $\alpha_1$, defined for \eqref{partialcoercivitysmalldata}, is independent of $\kappa$. Since $K_2-K_3 = K_3(\kappa-1)$, the discrete bilinear form of \eqref{penaltysecondderivative} becomes
\begin{align}
a(\vec{v}, \vec{v}) =& \tilde{a}(\vec{v}, \vec{v}) + \epsilon_1 K_3 \Big(2 \Ltwoinner{\vec{v} \cdot \curl \vec{v}}{\kdirector \cdot \curl \kdirector}{\Omega} + 2\Ltwoinner{\kdirector \cdot \curl \vec{v}}{\vec{v} \cdot \curl \kdirector}{\Omega} \nonumber \\
& \qquad +\Ltwoinner{\vec{v} \cdot \curl \kdirector}{\vec{v} \cdot \curl \kdirector}{\Omega} \Big) \label{avvsmalldata1}.
\end{align}
Observe that from \eqref{partialcoercivitysmalldata},
\begin{align}
(\alpha_1- 2\zeta \vert \alpha -1 \vert) \leq \tilde{a}(\vec{v}, \vec{v}) +\epsilon_1 K_3 \Ltwoinner{\vec{v} \cdot \curl \kdirector}{\vec{v} \cdot \curl \kdirector}{\Omega}. \label{partialcoercivitysmalldatacase1}
\end{align}
Consider the magnitude of the terms in \eqref{avvsmalldata1} not bounded from below in \eqref{partialcoercivitysmalldatacase1}, denoted as $\mathcal{G}(\vec{v}, \vec{v})$. As in the proof of Lemma 3.8 in \cite{Emerson1},
\begin{align} \label{case1Gbound}
\vert \mathcal{G}(\vec{v}, \vec{v}) \vert \leq& \epsilon_1 \alpha_3 \Hdcnorm{\vec{v}}{\Omega}^2,
\end{align}
where $\alpha_3$ is a constant defined in \cite{Emerson1}. Utilizing \eqref{partialcoercivitysmalldatacase1} and \eqref{case1Gbound},
\begin{align*}
a(\vec{v}, \vec{v}) \geq \alpha_1 \Hdcnorm{\vec{v}}{\Omega}^2 - 2\zeta \vert \alpha -1 \vert \Hdcnorm{\vec{v}}{\Omega}^2 - \epsilon_1 \alpha_3 \Hdcnorm{\vec{v}}{\Omega}^2 = (\beta_1 -\epsilon_1 \alpha_3) \Hdcnorm{\vec{v}}{\Omega}^2.
\end{align*}
It is, thus, sufficient to have $\epsilon_1 < \beta_1/\alpha_3$, guaranteeing that $(\beta_1 -\epsilon_1 \alpha_3)>0$.}

\case{$\kappa=1-\epsilon_2>0$, for $\epsilon_2>0$, and $K_1<K_3$.}
{Since $\kappa<1$, $\eta = 1+(\kappa-1)\beta=(1-\epsilon_2 \beta)$. For $K_1 < K_3$, there exists an $\epsilon_2$ small enough, such that $K_1 < (1-\epsilon_2 \beta)K_3$. This implies that, for small enough $\epsilon_2$,
\begin{equation*}
\alpha_1 = \frac{\min(K_1, (1-\epsilon_2 \beta)K_3)}{(C+1)} = \frac{K_1}{(C+1)}.
\end{equation*}
Therefore, $\alpha_1$ is again independent of $\kappa$. Since $K_2-K_3 = K_3 (\kappa-1)$, the discrete bilinear form of \eqref{penaltysecondderivative} becomes
\begin{align}
a(\vec{v}, \vec{v}) =& \tilde{a}(\vec{v}, \vec{v})- \epsilon_2 K_3 \Big(2 \Ltwoinner{\vec{v} \cdot \curl \vec{v}}{\kdirector \cdot \curl \kdirector}{\Omega}+2\Ltwoinner{\kdirector \cdot \curl \vec{v}}{\vec{v} \cdot \curl \kdirector}{\Omega}\nonumber \\
&\qquad +\Ltwoinner{\vec{v} \cdot \curl \kdirector}{\vec{v} \cdot \curl \kdirector}{\Omega} \Big) \label{avvsmalldatacase2}.
\end{align}
The terms of \eqref{avvsmalldatacase2}, not already bounded from below in \eqref{partialcoercivitysmalldata}, are bounded, utilizing \cite[Lemma 3.8]{Emerson1}, as
\begin{align} \label{case2Gbound}
\vert \mathcal{G}(\vec{v}, \vec{v}) \vert \leq \epsilon_2 \alpha_4  \Hdcnorm{\vec{v}}{\Omega}^2,
\end{align}
where $\alpha_4$ is a constant defined in \cite{Emerson1}. Using \eqref{partialcoercivitysmalldata} and \eqref{case2Gbound} implies that
\begin{equation*}
a(\vec{v}, \vec{v}) \geq \alpha_1 \Hdcnorm{\vec{v}}{\Omega}^2 - 2\zeta \vert \alpha -1 \vert \Hdcnorm{\vec{v}}{\Omega}^2 - \epsilon_2 \alpha_4 \Hdcnorm{\vec{v}}{\Omega}^2 \geq (\beta_1-\epsilon_2 \alpha_4) \Hdcnorm{\vec{v}}{\Omega}^2.
\end{equation*}
Thus, possibly requiring $\epsilon_2$ to be even smaller, choose $\epsilon_2 < \beta_1/\alpha_4$, so that $(\beta_1-\epsilon_2 \alpha_4)>0$.
 
In the case that $\kappa<1$, the additional restriction that $\beta < \frac{1}{1-\kappa}$ for $\vec{Z}$ to be USPD is necessary, which implies that $\epsilon_2\beta < 1$ is required. Therefore, for any fixed $\beta$, $\epsilon_2$ must also be taken small enough to satisfy this condition. Hence,
\begin{equation*}
\epsilon_2 < \min \left (\frac{\beta_1}{\alpha_4}, \frac{K_3 - K_1}{\beta K_3}, \frac{1}{\beta} \right).
\end{equation*}
}
\case{$\kappa=1-\epsilon_2>0$, for $\epsilon_2>0$, and $K_3 \leq K_1$.}{Here, again, $\eta = (1-\epsilon_2 \beta)$. For this case, it is clear that $(1-\epsilon_2 \beta)K_3 < K_1$. Thus,
\begin{equation*}
\alpha_1 = \frac{(1-\epsilon_2 \beta) K_3}{(C+1)}.
\end{equation*}
Using the same $\alpha_4$ as in the previous case and similar arguments,
\begin{align*}
a(\vec{v}, \vec{v}) &\geq \alpha_1 \Hdcnorm{\vec{v}}{\Omega}^2 - 2\zeta \vert \alpha -1 \vert \Hdcnorm{\vec{v}}{\Omega}^2 - \epsilon_2 \alpha_4 \Hdcnorm{\vec{v}}{\Omega}^2 \nonumber \\
&= \left (\frac{K_3}{C+1} - 2\zeta \vert \alpha -1 \vert - \frac{\epsilon_2 \beta K_3}{C+1} - \epsilon_2 \alpha_4 \right ) \Hdcnorm{\vec{v} }{\Omega}^2 \nonumber \\
& \geq \left( \beta_1 - \frac{\epsilon_2 \beta K_3}{C+1} - \epsilon_2 \alpha_4 \right) \Hdcnorm{\vec{v} }{\Omega}^2.
\end{align*}
Hence, in order for $\left( \beta_1 - \frac{\epsilon_2 \beta K_3}{C+1} - \epsilon_2 \alpha_4 \right)>0$ to hold, it is necessary that
\begin{equation*}
\epsilon_2 < \frac{\beta_1 (C+1)}{K_3 \beta + \alpha_4 (C+1)}.
\end{equation*}
Finally, $\epsilon_2$ must still be chosen sufficiently small with respect to $\beta$ such that $\epsilon_2 \beta < 1$, as in Case 2. Therefore,
\begin{equation*}
\epsilon_2 < \min \left (\frac{\beta_1 (C+1)}{K_3 \beta + \alpha_4 (C+1)}, \frac{1}{\beta} \right).
\end{equation*}
}
\end{caseof}
\end{proof} \\
\noindent Although the bounds on $\epsilon_1$ and $\epsilon_2$ are complicated, the only constant therein that depends on $\zeta$ is $\alpha_2$. The remainder of the constants are independent of $\zeta$.

The above lemmas allow for the formulation of the following summary theorem. 
\begin{theorem}
Under Assumption \ref{secass2}, if the conditions of Lemma \ref{coercivityauv} or Lemma \ref{coercivitysmalldata} are satisfied, the discrete variational problem in \eqref{auvform} is well-posed.
\end{theorem}
\begin{proof}
Lemmas \ref{boundedlinearform} and \ref{boundedbilinearform} imply that $F(\vec{v})$ and $a(\vec{u}, \vec{v})$ are continuous, respectively. Lemmas \ref{coercivityauv} or \ref{coercivitysmalldata} imply that $a(\vec{u}, \vec{v})$ is coercive. Therefore, by the Lax-Milgram Theorem \cite{Braess1}, \eqref{auvform} is a well-posed discrete variational problem.
\end{proof}

Therefore, the discretization of the linearizations arising in the penalty method are always well-posed under the assumption of small anisotropy in the system coefficients and sufficient conformance to the unit-length constraint. However, the penalty parameter must be chosen appropriately to achieve accurate representation of the unit-length constraint. If $\zeta$ is too small, constraint conformance becomes poor and the functional minimum does not accurately represent the constrained minimum. Alternatively, if $\zeta$ is too large, the solvability of the intermediate variational systems degrades in two possible ways. The neighborhood admitting coercivity around $\kappa=1$ shrinks and the system becomes increasingly ill-conditioned due to a decreasing coercivity constant, or coercivity is lost entirely and the possibility of non-invertible matrices appears. On the other hand, the proof of such well-posedness does not require establishing an inf-sup condition that necessitates subtle choices of finite-element spaces, as used for the Lagrange multiplier approach in \cite{Emerson1, Emerson2}.

\section{Robust Newton Step Methods} \label{TrustRegionMethods}

The Newton method applied to the Lagrange multiplier formulation discussed in \cite{Emerson1, Emerson2} employs na\"{i}ve incomplete Newton stepping. That is, for a computed Newton update direction, $\ddirector$, a constant damping factor, $0 < \omega \leq 1$, is applied such that the new iterate is given as $\director_{k+1} = \kdirector + \omega \ddirector$. Such an approach aims to improve convergence robustness when dealing with an inaccurate initial guess on coarse grids. However, this procedure may miss opportunities to take larger steps in ``good'' descent directions that effectively reduce the free energy.

Trust-region techniques are specifically designed to improve the robustness and efficiency of iterative procedures such as Newton's method. Updates are confined to a neighborhood, known as a trust region, where the accuracy of the linearized first-order optimality conditions is ``trusted''. These neighborhoods are expanded or contracted based on a measure of the model fidelity for a computed update. Significant research has produced both theoretical support and practical applications of such techniques \cite{Nocedal1}. This section discusses the use of constrained and unconstrained trust-region methods for the Lagrangian and penalty approaches discussed in Section \ref{energymodels}. For a general overview of trust-region methods, see \cite{Nocedal1}.

\subsection{Trust-Region Approaches for the Penalty Formulation}

Using the penalty functional in \eqref{Dirichletpenaltyfunctional}, the desired energy minimization is unconstrained. For this subsection, we denote the discretized forms of $\pd{}{\director} \left( \mathcal{P}_{\director}(\kdirector)[\vec{v}] \right)[\ddirector]$ and $\mathcal{P}_{\director}(\kdirector)[\vec{v}]$ as $U_k$ and $\vec{f}_k$, respectively. The quadratic model of the penalty functional, for a given $\kdirector$, is written
\begin{align} \label{quadraticModel}
M_k(\ddirector) = \mathcal{P}(\kdirector) + \vec{f}_k^T \ddirector + \frac{1}{2} \ddirector^T U_k \ddirector.
\end{align} 
As a consequence of the well-posedness theory developed in Section \ref{penaltywellposedness}, the matrix $U_k$ is positive definite for each iteration. Therefore, we follow the methodology in \cite{Byrd1, Shultz1}, computing steps by solving a trust-region minimization problem.

We seek an efficiently computable correction, $\ddirector$, that approximately minimizes the model in \eqref{quadraticModel}. In the following, we introduce two approaches to computing step length and direction for this problem. The performance of these techniques is vetted in the numerical experiments below.

Incomplete Newton stepping is equivalent to taking a small step in the descent direction, $-U_k^{-1}\vec{f}_k$. This is an effective means of finding energy minimizing solutions for both the penalty and Lagrangian methods. Therefore, in the first approach, a simple step selection technique is used in which the step is chosen satisfying the constrained minimization problem
\begin{align} \label{simplestepselection}
\ddirector(\Delta_k) &= \argmin \{ \mathcal{P}(\kdirector) + \vec{f}_k^T\ddirector+ \frac{1}{2} \ddirector^T U_k \ddirector : \nonumber \\
& \qquad \qquad \qquad \qquad \qquad \qquad \ltwonorm{\ddirector} \leq \Delta_k, \text{ } \ddirector = \mu U_k^{-1}\vec{f}_k \},
\end{align}
where $\Delta_k$ indicates the trust-region radius for iterate $\kdirector$. Candidate solutions of $\eqref{simplestepselection}$ are easily computed to be $-U_k^{-1}\vec{f}_k$, the full Newton step, which may or may not be inside the trust region, and $\pm \frac{\Delta_k}{\ltwonorm{U_k^{-1}\vec{f}_k}} U_k^{-1}\vec{f}_k$, representing steps to the trust region boundary.

An important aspect of trust-region methods is the adjustment of the trust-region radius and application of a computed step. This typically involves a measure of a computed step's merit. For a computed step, $\ddirector$, we compute the ratio,
\begin{align*}
\rho_k = \frac{\mathcal{P}(\kdirector) - \mathcal{P}(\kdirector + \ddirector)}{M_k(\vec{0}) - M_k(\ddirector)},
\end{align*}
of the actual to the predicted reduction in $\mathcal{P}$ due to the computed step. The closer $\rho_k$ is to $1$, the more accurately the quadratic model behavior matches that of the true functional.

If the ratio, $\rho_k$, is deemed acceptable, the step is applied and the trust region expands, remains static, or shrinks depending on the specific value of $\rho_k$. If $\rho_k$ is too small, the step is rejected, the trust-region radius is shrunk, and the process repeated. To quantify, let $0 < \eta_3 < \eta_1 < \eta_2$ be positive constants, along with $0 < C_1 < 1 < C_3$. Further, let $\bar{\Delta}$ be a maximum limit on the trust-region size. Using these parameters, the specific decision trees for accepting the step, and subsequently adjusting the trust region, are given in Procedures \ref{SolutionAcceptanceTree} and \ref{TRRegionTree}, respectively.

\noindent
\begin{minipage}[t]{.49 \textwidth}
\vspace{0pt}
\begin{procedure}[H] \label{SolutionAcceptanceTree}
\uIf{$\rho_k > \eta_3$}{Accept step: $\director_{k+1} = \kdirector + \ddirector$.}
\Else{Reject step: $\director_{k+1} = \kdirector$.}
\caption{Solution update().}
\end{procedure}
\end{minipage}
\begin{minipage}[t]{.49 \textwidth}
\vspace{0pt}
\begin{procedure}[H] \label{TRRegionTree}
\uIf{$\rho_k < \eta_1$}{Shrink region: $\Delta_{k+1} = C_1 \Delta_k$.}
\uElseIf{$\rho_k > \eta_2$ and $\ltwonorm{\ddirector} = \Delta_k$}{Expand region: $\Delta_{k+1} = \min (C_3 \Delta_k, \bar{\Delta})$.}
\Else{Keep region constant: $\Delta_{k+1}=\Delta_k$.}
\caption{TR size adjustment().}
\end{procedure}
\end{minipage}
\vspace{0.2cm}

For our algorithm, if the components of the ratio, $\rho_k$, are very small and the computed step lies on the interior of the trust region, representing a full step towards satisfying the first-order optimality conditions, we choose to apply the step regardless of $\rho_k$ and the trust region remains static. In this way, the trust-region minimization approach is used until we trust in the application of full Newton steps to obtain the first-order optimality conditions. A set of typical values for the trust-region parameters discussed above are listed in Table \ref{penaltyTRConstants} and used in the numerical methods below. 

A number of well-founded techniques improving trust-region step selection exist, including dogleg and two-dimensional (2D) subspace methods \cite{Byrd1, Fletcher1, Sorensen1, Nocedal1}. Because the 2D-subspace method subsumes both the simple step selection approach above and dogleg methods, it is chosen as the alternative step selection computation here. Steps are computed by solving 
\begin{align} \label{stepselectionByrd}
\ddirector (\Delta_k) &= \argmin \{ \mathcal{P}(\kdirector) + \vec{f}_k^T \ddirector + \frac{1}{2} \ddirector^T U_k \ddirector : \nonumber \\
& \qquad \qquad \qquad \qquad \qquad \qquad \ltwonorm{\ddirector} \leq \Delta_k, \text{ } \ddirector = \mu_1 \vec{f}_k + \mu_2 U_k^{-1}\vec{f}_k \}. 
\end{align}
Again, the candidate solutions for \eqref{stepselectionByrd} are efficiently computable, amounting to solving for the zeroes of a fourth-order polynomial. 

\subsubsection{A Renormalization Penalty Method}

In addition to the standard penalty method discussed above, a modification is also considered in the numerical experiments below. Once the approximation to the solution has been updated with a computed and accepted step, the new approximation is renormalized at the finite-element nodes. That is, the updated approximation is projected onto the unit sphere at each finite-element node. This procedure is similar to that presented in \cite{Ramage3} for a Lagrange multiplier formulation. There, the approach is derived within a nullspace method framework using the one constant approximation. Here, renormalization is applied to the penalty method, with and without trust regions and nested iteration, for anisotropic Frank constants.

This renormalization aims at improving unit-length conformance for solutions computed by the penalty method. The expectation is that this will lead to enhanced constraint conformance at lower penalty weights. However, unless the renormalization scaling is relatively uniform across nodes, the Newton direction may be significantly altered. Throughout this paper, this modification will be referred to as the ``renormalization'' penalty method.

\subsection{A Trust-Region Method for the Lagrange Multiplier Formulation}

Applications of trust-region techniques to optimization problems with nonlinear constraints have also been developed. However, certain challenges arise in the theory and practical use of such methods \cite{Maratos1}. Here, we consider existing trust-region approaches in the context of finite-element methods. For this subsection, let $W_k$ be the matrix associated with a finite-element discretization of the second-order derivative of \eqref{functional2} (i.e., the functional without the Lagrange multiplier term), given by $\pd{}{\director} \left( \mathcal{F}_{\director}(\kdirector)[\vec{v}] \right)[\ddirector]$. For the trust-region approach, write the constraint
\begin{equation} \label{constraintC}
c(\director) = \Ltwoinner{\director \cdot \director -1}{\director \cdot \director -1}{\Omega} = 0.
\end{equation}
The G\^{a}teaux derivative of \eqref{constraintC} is
\begin{equation}\label{gradientC}
\pd{}{\director} c(\director)[\vec{v}] = 4 \Ltwoinner{\director \cdot \director -1}{\director \cdot \vec{v}}{\Omega}.
\end{equation}
Finally, let $\vec{c}_k$ be the column vector representing the finite-element discretized form of \eqref{gradientC} at iterate $\kdirector$.  

One of the significant advantages of finite-element discretizations is the inherent sparsity of the resulting matrices. Trust-region algorithms in the Byrd-Omojokun family \cite{Su1, Byrd2, Omojokun1} require computation of the generally non-sparse matrix $N_k$, whose columns form an orthonormal basis for the orthogonal complement of $\vec{c}_k$, as well as the formation and inversion of the matrix $N_k^T W_k N_k$. In general, the matrix $N_k^T W_k N_k$ is not sparse and quite large, as $W_k$ has dimension $m \times m$ and $N_k$ is $m \times (m-1)$, where $m$ is the number of discretization degrees of freedom for $\director$. Storage and computation with these dense matrices proves to be prohibitive, even on relatively small grids. Therefore, any advantages garnered by the use of these trust regions is outweighed by loss of the finite-element sparsity. Similarly, trust-region methods based on the fundamental work in \cite{Vardi1} suffer from sparsity fill-in issues for large matrices in the context of finite-element methods.

To preserve sparsity properties, while still maintaining some advantages of a trust-region approach, we implement a simple trust-region method specifically fitted to the Lagrange multiplier formulation of the minimization problem. For the Lagrange multiplier approach in Section \ref{LagrangeMultiplierApproach}, we compute a Newton update direction as in \cite{Emerson1}, $\delta \chi = [\ddirector \text{ } \dlambda]^T$. This update is meant to bring $\kdirector$ and $\klambda$ closer to satisfying the first-order optimality conditions. Let $\mathcal{L}_0(\director_k, \lambda_k)$ represent the finite-element discretized form of the right-hand side of Equation \eqref{newtonhessian} for $\director_k$ and $\lambda_k$. Define the proportions $w_k$ and $w_{\text{lim}}$, such that $0 < w_{\text{lim}} \leq w_k \leq 1$, where $w_{\text{lim}}$ is a lower bound for $w_k$. For a given step, $w_k \delta \chi$, the expected change in $\ltwonorm{\mathcal{L}_0(\kdirector, \klambda)}$ is equal to $w_k \ltwonorm{\mathcal{L}_0(\kdirector, \klambda)}$. Therefore, we define the ratio
\begin{equation*}
\rho_k = \frac{\ltwonorm{\mathcal{L}_0(\kdirector, \klambda)} - \ltwonorm{\mathcal{L}_0(\kdirector+w_k \ddirector, \klambda+w_k \dlambda)}}{w_k \ltwonorm{\mathcal{L}_0(\kdirector, \klambda)}}.
\end{equation*}
This ratio compares the change in the Lagrangian predicted by the linearized model to the actual change in the true Lagrangian for a computed step. 

Let $0 < \eta_2 < \eta_1$ and $w_{\text{inc}}, w_{\text{dec}} \in (0, 1]$. Since $w_k$ is a scaling factor, rather than a radius length, step selection and trust-region adjustment differ slightly from the procedures discussed above and are given in Procedures \ref{SolutionAcceptanceTreeLag} and \ref{TRRegionTreeLag}, respectively. 

\noindent
\begin{minipage}[t]{.49 \textwidth}
\vspace{0pt}
\begin{procedure}[H] \label{SolutionAcceptanceTreeLag}
\uIf{$\rho_k > \eta_2$ or $w_k = w_{\text{lim}}$}{Accept step: $[\director_{k+1}\text{ }\lambda_{k+1}]^T = [\kdirector \text{ } \klambda]^T + w_k \delta \chi$.}
\Else{Reject step: $[\director_{k+1}\text{ }\lambda_{k+1}]^T = [\kdirector \text{ } \klambda]^T$.}
\caption{Solution update().}
\end{procedure}
\end{minipage}
\begin{minipage}[t]{.49 \textwidth}
\vspace{0pt}
\begin{procedure}[H] \label{TRRegionTreeLag}
\uIf{$\rho_k < \eta_2$}{Shrink region:  $w_{k+1} = \max(w_{\text{lim}}, w_k - w_{\text{dec}})$.}
\uElseIf{$\eta_2 < \rho_k < \eta_1$}{Keep region constant: $w_{k+1} = w_k$.}
\Else{Expand region: $w_{k+1} = \min(w_k + w_\text{inc}, 1)$.}
\caption{TR size adjustment().}
\end{procedure}
\end{minipage}
\vspace{0.2cm}

\section{Numerical Results} \label{nummethodology}

In this section, we compare the performance of the methods outlined above for three benchmark equilibrium problems. The general algorithm utilized by each method has three stages; see Algorithm \ref{generalAlgorithm}. The outer stage implements nested iteration (NI) \cite{Starke1} where, at each level, the approximation to the solution is iteratively updated. These updates are computed via one of the methods described above. In general, the iteration stopping criterion, on a given level, is based on a set tolerance for the approximation's conformance to the first-order optimality conditions in the standard Euclidean $l_2$-norm. For the renormalization penalty method, the Newton iteration tolerance is based on the reduction of the ratio of the energy from the previous step to the current step rather than conformance to the first-order optimality conditions. In the numerical experiments carried out below, both tolerances were held at $10^{-4}$. The approximate solution is then transferred to a finer grid. In the current implementation, these finer grids represent uniform refinements of the initial coarse grid. However, adaptive refinement could also be performed.

The components of the variational problems in Equations \eqref{newtonhessian} and \eqref{penaltylinearization} are discretized with finite elements on each grid.   Both formulations use $Q_2 \times Q_2 \times Q_2$ elements for $\director$, while the Lagrange multiplier approach uses $P_0$ elements for $\lambda$, as in \cite{Emerson1, Emerson2}. In this section, the arising matrices are inverted using the UMFPACK LU decomposition \cite{TADavis1,TADavis2, TADavis3, TADavis4}. In Section \ref{BraessSarazinSmoothing}, we introduce an optimally scaling multigrid method with improved time to solution. The algorithm's discretizations and grid management are performed with the widely used deal.II finite-element and scientific computing library \cite{BangerthHartmannKanschat2007, DealIIReference}.

\vspace{.3cm}
\begin{algorithm}[H] \label{generalAlgorithm}
\SetAlgoLined
~\\
0. Initialize solution approximation on coarse grid.
~\\
\While{Refinement limit not reached}
{
	\While{Nonlinear iteration tolerance not satisfied}
	{
		1. Assemble discrete components of System \eqref{newtonhessian} or \eqref{penaltylinearization} on current grid, $H$. ~\\
		2. Compute correction to current approximation. ~\\
		3. Update current approximation. ~\\
	}
	4. Uniformly refine the grid to size $h$. ~\\
	5. Interpolate solution $\vec{u}_{H} \to \vec{u}_{h}$.
}
\caption{General minimization algorithm with NI}
\end{algorithm}
\vspace{.3cm}

Each of the problems below is posed on a unit-square domain in the $xy$-plane, such that $\Omega = \{ (x,y) \text{ } \vert \text{ } 0 \leq x,y \leq 1 \}$. It is assumed that this domain represents a uniform slab. That is, the vector $\director$ may have nonzero $z$-component but $\pd{\director}{z} = \vec{0}$. Dirichlet boundary conditions are applied at the $y$-edges and periodic boundary conditions are assumed at the boundaries $x=0$ and $x=1$. The experiments to follow consider an $8 \times 8$ coarse mesh ascending in $6$ uniform refinements to a $512 \times 512$ mesh. 

For the numerical experiments, each of the trust-region methods discussed above is applied. For the penalty trust-region methods, the initial trust region radius is set to $\Delta_{\text{init}}$. At each refinement level, the trust-region radius is reset to $\Delta_{\text{init}}$ plus an incremental increase, $\Delta_{\text{inc}}$, with a maximum of $\bar{\Delta}$. The Lagrangian trust-region approach sets the initial value of $w_k$ to $w_{\text{init}}$. After each refinement, $w_k$ is reset to $w_{\text{init}}$ plus $w_{\text{lev}}$, up to a maximum of $1$. These increments are due to the increasing accuracy of the iterates at each grid level. These constants are outlined in Tables \ref{penaltyTRConstants} and \ref{LagrangianTRConstants}

\begin{table}[h!]
\centering
{\small
\begin{tabular}{|c|c|c|c|}
\hline
$\eta_1 = 0.25$ & $\eta_2 = 0.75$ & $\eta_3 = 0.125$ & $C_1 = 0.5$ \\
\hline
$C_3 = 1.3$ & $\Delta_{\text{inc}} = 0.3$ & \rule{0pt}{1.1em}$\bar{\Delta} = 100$ & $ \Delta_{\text{init}} = 0.3$ \\
\hline
\end{tabular}
}
\caption{\small{Trust-region parameters for the penalty formulation.}}
\label{penaltyTRConstants}
\end{table}

\begin{table}[h!]
\centering
{\small
\begin{tabular}{|c|c|c|c|}
\hline
$\eta_1 = 0.5$ & $\eta_2 = 0.25$ & $w_{\text{inc}} = 0.1$ & $w_{\text{dec}} = 0.1$ \\
\hline
$w_{\text{lev}} = 0.1$ & $w_{\text{min}} = 0.1$ & $w_{\text{init}} = 0.2$ & $-$ \\
\hline
\end{tabular}
}
\caption{\small{Trust-region parameters for the Lagrangian formulation.}}
\label{LagrangianTRConstants}
\end{table}

The non-trust-region, incomplete Newton stepping approach is also performed for both formulations as a comparison benchmark with an initial $\omega = 0.2$, increasing by $0.2$ at each refinement to a maximum of $1$. The performance of each of these methods is then compared. In the results to follow, all reported free energies are computed using only the free elastic quantities without any augmentations, such as the penalty terms.

\subsection{Twist Equilibrium Configuration}

The first set of boundary conditions induce a classical twist equilibrium configuration \cite{Stewart1}. For this experiment, and the tilt-twist experiment in the next subsection, let the general form of the solution be
\begin{equation} \label{sphericalRepresentation}
\director = \big(\cos (\theta(y)) \cos (\phi(y)), \cos (\theta(y)) \sin (\phi (y)), \sin (\theta(y))\big).
\end{equation}
Note that the known analytical solutions have a one-dimensional structure, but the numerical experiments below are full two-dimensional simulations. For the twist configuration, let $\theta_0 = \frac{\pi}{8}$. At the boundaries $\theta(0)=-\theta_0$, $\theta(1) = \theta_0$, and $\phi(0) = \phi(1)=0$. The Frank constants for this problem are $K_1 = 1.0$, $K_2 = 1.2$, and $K_3 = 1.0$. The analytical equilibrium solution for these boundary conditions and Frank constants is derived, under a rotated coordinate system, in \cite{Stewart1}. The solution is given by
\begin{equation*}
\director = (\cos (\theta_0(2y-1)), 0, \sin(\theta_0(2y-1))),
\end{equation*}
with true free-elastic energy $2K_2 \theta_0^2$. This corresponds to an expected free energy of $0.37011$. The existence of an analytical solution for this problem allows for the computation of an $L^2$-error for each computed approximation. 

\begin{table}[h!]
\centering
{\small
\begin{tabular}{|c|c|c|c|c|c|c|}
\hline
Type & Free Energy & $L^2$-error & Min. Dev. & Max Dev. & Cost  & TR Cost\\
\hlinewd{1.3pt}
Lagrangian & $0.370110$ & $2.076$e-11& $-1.43$e-14 & $7.00$e-15 & $1.350$ & $1.340$\\
\hlinewd{1.3pt}
Pen. $\zeta = 10^1$ & $0.358832$ & $1.589$e-02 & $-3.96$e-02 & $-3.59$e-05 & $1.371$ & $1.354$ \\
\hline
Pen. $\zeta = 10^2$ & $0.368481$ & $1.993$e-03 & $-4.32$e-03 & $-1.16$e-05 & $1.376$ & $1.355$\\
\hline
Pen. $\zeta = 10^3$ & $0.369931$ & $2.107$e-04 & $-4.32$e-04 & $-3.68$e-06 & $1.440$ & $1.418$ \\
\hline
Pen. $\zeta = 10^4$ & $0.370092$ & $2.143$e-05 & $-4.32$e-05 & $-1.14$e-06 & $1.448$ & $1.420$ \\
\hline
Pen. $\zeta = 10^5$ & $0.370108$ & $2.154$e-06 & $-4.32$e-06 & $-3.32$e-07 & $1.447$ & $1.426$ \\
\hline
Pen. $\zeta = 10^6$ & $0.370110$ & $2.157$e-07 & $-4.32$e-07 & $-7.27$e-08 & $-$ & $1.436$\\
\hline
Pen. $\zeta = 10^7$ & $0.370110$ & $2.158$e-08 & $-5.05$e-08 & $-9.98$e-09 & $-$ & $1.465$\\
\hline
Pen. $\zeta = 10^8$ & $0.370110$ & $2.158$e-09 & $-5.18$e-09 & $-1.06$e-09 & $-$ & $1.516$\\
\hline
Pen. $\zeta = 10^9$ & $0.370110$ & $2.168$e-10 & $-5.19$e-10 & $-1.06$e-10 & $-$ & $1.639$\\
\hline
\end{tabular}
}
\caption{\small{Statistics for the twist equilibrium solution with the different formulations and penalty weights. Included is the system free energy, the computed $L^2$-error on the finest grid, and the minimum and maximum deviations from unit director length at the quadrature nodes. Approximations of the cost in WUs for the corresponding method with no trust regions and simple trust regions are included. Dashes in the columns indicate divergence.}}
\label{TwistBCFormulationComparisons}
\end{table}

Table \ref{TwistBCFormulationComparisons} compares the performance of the Lagrange multiplier method to the penalty method without renormalization. The runs were performed with nested iteration and the approximate work, measured in terms of assembling and solving a single linearization step on a $512 \times 512$ grid, referred to as a work unit (WU), is given for the corresponding method with no trust regions and the simple trust region approaches, respectively. The work approximation is computed by summing the number of non-zeroes in each matrix across all grids and dividing by the number of non-zeroes in the (fixed) sparsity pattern at the finest level. The non-trust-region, incomplete Newton stepping discussed above diverged for penalty parameters of $\zeta = 10^6$ and greater. However, smaller damping parameters may yield convergence. Both penalty-method trust-region approaches converged without modification.

The table demonstrates the superior performance of the Lagrange multiplier method for this problem across all statistics with lower error, cost, and tighter conformance to the constraint. The penalty method does not match the free energy obtained by the Lagrangian formulation until reaching a penalty weight of $10^6$ and, without trust regions, encounters divergence issues for these large penalty weights. While trust regions do not significantly reduce overall computations costs, Table \ref{TwistBCFormulationComparisons} suggests that they significantly improve robustness. 

\begin{table}[h!]
\centering
{\small
\begin{tabular}{|c|c|c|c|c|c|c|}
\hline
& \multicolumn{2}{|c|}{No Trust Region} & \multicolumn{2}{|c|}{Simple Trust Region} & \multicolumn{2}{|c|}{2D Trust Region} \\
\hline
Type & $L^2$-error & Cost & $L^2$-error & Cost & $L^2$-error & Cost \\
\hline
Pen. $\zeta = 10^1$ & $1.457$e-02 & $1.338$ & $1.457$e-02  & $1.334$ & $1.457$e-02 & $1.334$ \\
\hline
Pen. $\zeta = 10^2$ & $8.932$e-05 & $1.338$ & $8.931$e-05 & $1.334$ & $8.931$e-05 & $1.334$ \\
\hline
Pen. $\zeta = 10^3$ & $3.358$e-06 & $1.339$ & $3.357$e-06 & $1.334$ & $3.357$e-06 & $1.335$ \\
\hline
Pen. $\zeta = 10^4$ & $1.523$e-07 & $1.340$ & $1.116$e-07 & $1.336$ & $1.116$e-07 & $1.336$ \\
\hline
Pen. $\zeta = 10^5$ & $6.260$e-08 & $8.113$ & $3.595$e-09 & $1.364$ & $3.592$e-09 & $1.340$ \\
\hline
Pen. $\zeta = 10^6$ & $6.356$e-06 & $81.120$ & $1.688$e-02 & $73.052$ & $1.098$e-07 & $2.731$ \\
\hline
\end{tabular}
}
\caption{\small{A comparison of renormalization penalty methods, with and without trust-region approaches, for the twist solution. For each algorithm, the computed $L^2$-error on the finest grid and an approximation of the cost in WUs is included.}}
\label{TwistBCFormulationRenormPenaltyTRs}
\end{table}

The results in Tables \ref{TwistBCFormulationRenormPenaltyTRs} and \ref{TwistBCFormulationComparisonsRenormPenalty} show the performance of the renormalization penalty method with and without trust regions. Table \ref{TwistBCFormulationComparisonsRenormPenalty} provides additional statistics for the 2D-subspace minimization trust-region approach discussed in Table \ref{TwistBCFormulationRenormPenaltyTRs}. For the twist equilibrium solution, the renormalization penalty method obtains better error values for smaller penalty weights than the unmodified penalty method. In Table \ref{TwistBCFormulationComparisonsRenormPenalty}, using the 2D-subspace minimization trust-region approach, we obtain an error of $3.592$e-$09$ with a penalty weight of only $\zeta = 10^5$. Moreover, the minimum and maximum deviation of the director at the quadrature nodes is closer to that of the Lagrangian method. However, the performance improvements rely more heavily on the penalty parameter. While an error measure closer to the Lagrange multiplier formulation is achieved for $\zeta = 10^5$, performance degrades at $\zeta = 10^6$, with notable jumps in costs for all methods recorded in Table \ref{TwistBCFormulationRenormPenaltyTRs}. The increases in error are due to the algorithm beginning to emphasize the unit-length constraint over proper director orientation. Correctly selecting the penalty weight represents a fundamental difficulty for this method.

\begin{table}[h!]
\centering
{\small
\begin{tabular}{|c|c|c|c|c|c|}
\hline
Type & Free Energy & $L^2$-error & Min. Dev. & Max Dev. & 2D TR Cost \\
\hline
Pen. $\zeta = 10^1$ & $0.370168$ & $1.457$e-02 & $-4.58$e-11 & $4.58$e-11 & $1.334$ \\
\hline
Pen. $\zeta = 10^2$ & $0.370111$ & $8.931$e-05 & $-1.68$e-11 & $1.68$e-11 & $1.334$ \\
\hline
Pen. $\zeta = 10^3$ & $ 0.370110$ & $3.357$e-06 & $-5.18$e-12 & $5.16$e-12 & $1.335$  \\
\hline
Pen. $\zeta = 10^4$ & $0.370110$ & $1.116$e-07 & $-1.45$e-12 & $1.43$e-12 & $1.336$ \\
\hline
Pen. $\zeta = 10^5$ & $0.370110$ & $3.592$e-09 & $-3.16$e-13 & $2.98$e-13 & $1.340$ \\
\hline
Pen. $\zeta = 10^6$ & $0.370110$ & $1.098$e-07 & $-4.04$e-14 & $2.20$e-14 & $2.731$ \\
\hline
\end{tabular}
}
\caption{\small{Statistics for the twist equilibrium solution with different penalty weights. Here, the penalty method with renormalization and 2D-subspace minimization is considered. Included is the system free energy, the computed $L^2$-error on the finest grid, the minimum and maximum deviations from unit director length at the quadrature nodes, and an approximation of the cost in WUs for the corresponding method.}}
\label{TwistBCFormulationComparisonsRenormPenalty}
\end{table}

Figure \ref{fig:TwistNIFig} displays the number of iterations required to reach the specified iteration tolerance within a nested iteration scheme alongside the final solution computed by the Lagrange multiplier formulation in Figure \ref{fig:TwistHelv}. Counts for both the Lagrange multiplier approach and penalty formulation, with and without renormalization, for a penalty parameter $\zeta = 10^3$ are shown. In general, the trust-region methods significantly reduce iteration counts on the coarse grids. However, on the finer grids, this reduction is not sustained due to the efficiency of nested iteration. Because the improved iteration counts are confined to the coarsest grids, overall cost reduction is generally small. For example, the approximate cost for the Lagrange multiplier method was reduced very slightly from $1.350$ WUs to $1.340$ WUs, only resulting in a one second drop in overall time to solution.

\begin{figure}[h!]
\centering
\begin{subfigure}[b]{.49 \textwidth}
\raggedleft
  \includegraphics[scale=.32]{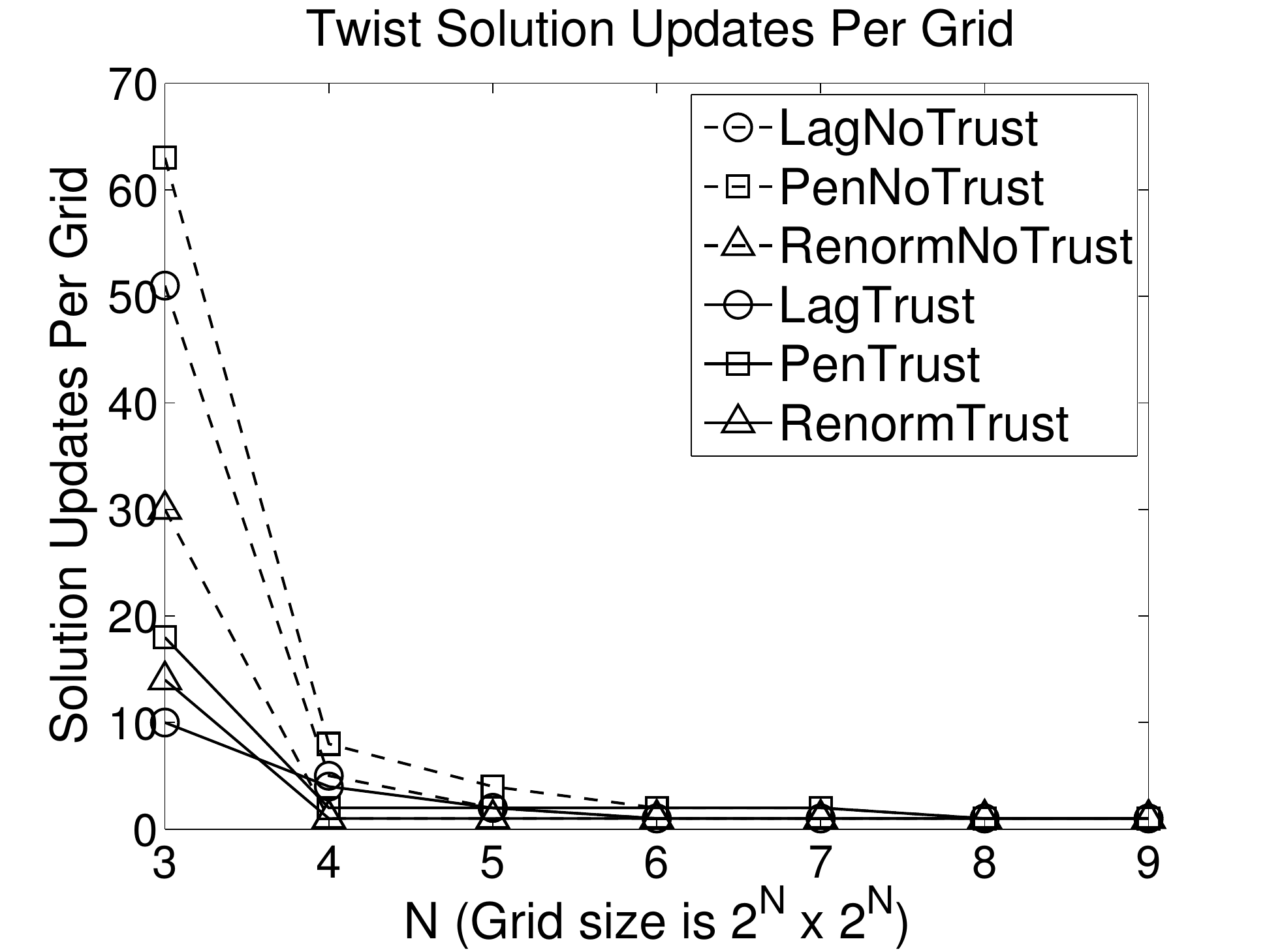}
  \caption{}
  \label{fig:TwistNIFig}
\end{subfigure}
\begin{subfigure}[b]{.49 \textwidth}
\raggedright
  \includegraphics[scale=.31]{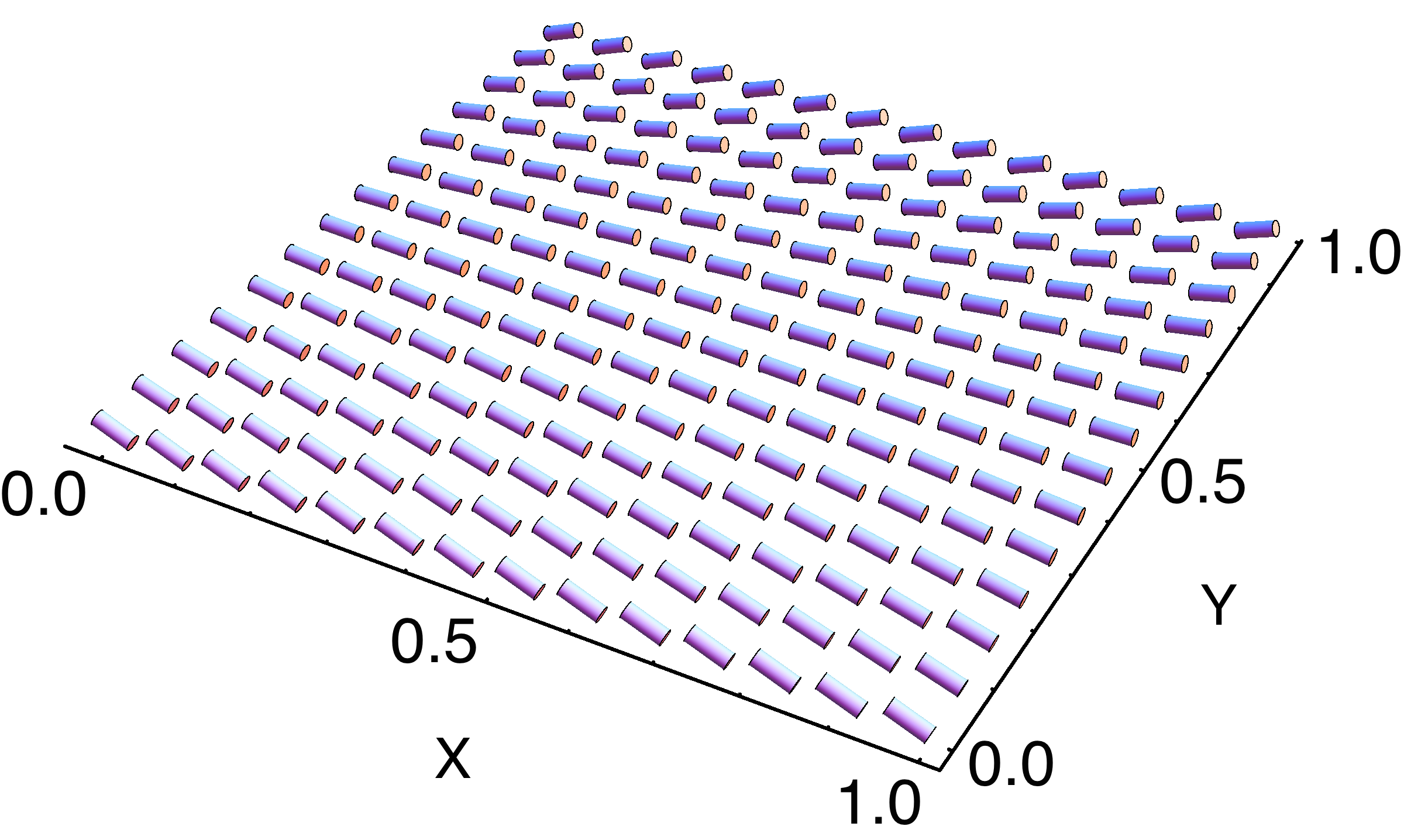}
  \caption{}
  \label{fig:TwistHelv}
\end{subfigure}
\caption{\small{(\subref{fig:TwistNIFig}) Number of iterations required to reach iteration tolerance for each method with NI. The penalty weight for the penalty formulation was $\zeta = 1000$. Only the 2D-subspace minimization trust-region approach is displayed, as the behavior of simple trust regions is similar. (\subref{fig:TwistHelv}) The final computed solution for the Lagrangian formulation on a $512 \times 512$ mesh (restricted for visualization).}}
\label{TwistNIFigure}
\end{figure}

\begin{table}[h!]
{\small
\begin{minipage}[t]{.49 \linewidth}
\begin{tabular}[t]{|c|c|c|c|}
\hline
\multicolumn{4}{|c|}{Lagrangian}\\
\hlinewd{1.3pt}
\multicolumn{2}{|c|}{Method} & Solve Cost & Run Time\\
\hline
No NI & No TR & $61$ & $17,975$s \\
\hline
NI & No TR& $1.350$ & $550$s \\
\hline
No NI &TR & $10$ & $3,071$s  \\
\hline
NI & TR & $1.340$ & $548$s \\
\hlinewd{1.3pt}
\multicolumn{4}{|c|}{Renormalization Penalty: $\zeta = 10^5$}\\
\hlinewd{1.3pt}
\multicolumn{2}{|c|}{Method} & Solve Cost & Run Time\\
\hline
No NI & No TR & $38$ & $11,838$s\\
\hline
NI & No TR & $8.113$ & $2,272$s\\
\hline
No NI & TR & $29$ & $9,172$s \\
\hline
No NI & TR 2D & $32$ & $10,147$s \\
\hline
NI & TR & $1.364$ & $585$s \\
\hline
NI & TR 2D & $1.340$ & $584$s \\
\hline
\end{tabular}
\end{minipage}
\begin{minipage}[t]{.49 \linewidth}
\begin{tabular}[t]{|c|c|c|c|}
\hline
\multicolumn{4}{|c|}{Unmodified Penalty: $\zeta = 10^5$}\\
\hlinewd{1.3pt}
\multicolumn{2}{|c|}{Method} & Solve Cost & Run Time\\
\hline
No NI & No TR & $142$ & $41,013$s\\
\hline
NI & No TR & $1.474$ & $593$s\\
\hline
No NI & TR & $63$ & $18,425$s \\
\hline
No NI & TR 2D & $64$ & $19,287$s \\
\hline
NI & TR & $1.426$ & $569$s \\
\hline
NI & TR 2D & $1.424$ & $574$s \\
\hlinewd{1.3pt}
\multicolumn{4}{|c|}{Unmodified Penalty: $\zeta = 10^9$}\\
\hlinewd{1.3pt}
\multicolumn{2}{|c|}{Method} & Solve Cost & Run Time\\
\hline
No NI & No TR & $-$ & $-$\\
\hline
NI & No TR& $-$ & $-$\\
\hline
No NI & TR &$1016$ & $294,349$s \\
\hline
No NI & TR 2D & $1736$ & $511,874$s \\
\hline
NI & TR & $1.639$ & $641$s \\
\hline
NI &TR 2D & $1.958$ & $764$s \\
\hline
\end{tabular}
\end{minipage}
}
\caption{\small{Twist statistics comparison for NI and trust region combinations. The solve cost column displays an approximation of the work in WUs for the corresponding method. The overall time to solution is also presented. Dashes in the columns indicate divergence.}}
\label{TwistWUTable}
\end{table}

Table \ref{TwistWUTable} summarizes both the efficiency of nested iteration and highlights the strengths of certain applications of trust-region methods. For all of the constraint enforcement formulations, nested iteration offers very clear cost improvements. Coupling nested iteration with the Lagrange multiplier method for this problem is quite powerful, yielding the fastest overall run time and highest accuracy. Trust regions have a clear impact on time to solution in the absence of nested iteration but offer modest time to solution improvements when coupled with NI. 

If the penalty method is used, pairing nested iteration with trust regions increases robustness and cost consistency. For example, the use of trust regions for the unmodified penalty method with $\zeta = 10^9$ overcomes prominent divergence issues. In addition to the improved error performance in Table \ref{TwistBCFormulationRenormPenaltyTRs}, for $\zeta = 10^5$, the renormalization penalty method is generally faster than the unmodified penalty approach with the same penalty weight. The slightly slower overall run times when NI is paired with trust regions, in comparison with the unmodified penalty method, are due to the work involved in normalizing the director after each iteration. As discussed above, a shortcoming of the renormalization penalty method is sensitivity to parameter choice. 

\subsection{Tilt-Twist Equilibrium Configuration}

For this problem, $\director$ retains the form in \eqref{sphericalRepresentation} and the same boundary conditions are applied with $\theta_0 = \frac{\pi}{4}$ and Frank constants of $K_1 = 1.0$, $K_2 = 3.0$, and $K_3 = 1.2$. Twist solutions incorporating a nonplanar tilt deviating from parallel alignment with the $xz$-plane are investigated in \cite{Leslie2, Leslie3}. It is shown that nonplanar twist solutions become available at a computable threshold. This threshold is satisfied for the chosen parameters. The analytical, energy-minimizing, tilt-twist solution is defined implicitly for a rotated coordinate system in \cite{Stewart1, Leslie2}. The associated analytical, free-elastic energy for the chosen parameters is $3.59294$.

For the tilt-twist equilibrium solution, the incomplete Newton stepping approach converged for all of the penalty weights considered. Table \ref{TiltTwistBCFormulationComparisons} details the statistics for the unmodified penalty method compared with the Lagrange multiplier method. Again, the Lagrange multiplier method outperforms the penalty method in each category. The free energy of the Lagrange multiplier method is not obtained by the penalty method until $\zeta$ reaches $10^8$. 

It should be noted that the behavior of the error for the Lagrangian method, as well as the penalty method for weights greater than $10^7$, is affected by the implicit definition of the true solution. The analytical solution for the tilt-twist equilibrium configuration is implicitly defined by a complicated set of equations, which are solved approximately at the appropriate quadrature points using Mathematica. Solving these equations involves successive root finding for complicated integral equations where the unknowns are limits of integration. Approximation error creates an artificial limit for the computed error at accuracies smaller than $10^{-7}$.

\begin{table}[h!]
\centering
{\small
\begin{tabular}{|c|c|c|c|c|c|c|}
\hline
Type & Free Energy & $L^2$-error & Min. Dev. & Max Dev. & Cost & TR Cost\\
\hlinewd{1.3pt}
Lagrangian & $3.59294$ & $4.717$e-07& $-7.89$e-10 & $7.88$e-10 & $1.463$ & $1.447$\\
\hlinewd{1.3pt}
Pen. $\zeta = 10^1$ & $2.15620$ & $4.403$e-01 & $-4.78$e-01 & $-2.62$e-04 & $1.458$ & $1.333$\\
\hline
Pen. $\zeta = 10^2$ & $3.38037$ & $4.597$e-02 & $-5.01$e-02 & $-1.17$e-04 & $1.732$ & $1.665$\\
\hline
Pen. $\zeta = 10^3$ & $3.56953$ & $4.565$e-03 & $-4.97$e-03 & $-3.88$e-05 & $1.732$ & $1.665$ \\
\hline
Pen. $\zeta = 10^4$ & $3.59052$ & $4.606$e-04 & $-5.00$e-04 & $-1.21$e-05 & $2.735$ & $2.665$ \\
\hline
Pen. $\zeta = 10^5$ & $3.59269$ & $4.590$e-05 & $-5.00$e-05 & $-3.56$e-06 & $2.743$ & $2.667$ \\
\hline
Pen. $\zeta = 10^6$ & $3.59291$ & $4.253$e-06 & $-5.01$e-06 & $-8.05$e-07 & $2.782$ & $2.678$ \\
\hline
Pen. $\zeta = 10^7$ & $3.59293$ & $2.735$e-07 & $-5.83$e-07 & $-1.14$e-07 & $2.809$ & $2.723$ \\
\hline
Pen. $\zeta = 10^8$ & $3.59294$ & $4.340$e-07 & $-6.00$e-08 & $-1.22$e-08 & $2.885$ & $2.747$ \\
\hline
Pen. $\zeta = 10^9$ & $3.59294$ & $4.676$e-07 & $-6.01$e-09 & $-1.24$e-09 & $3.218$ & $2.879$ \\
\hline
\end{tabular}
}
\caption{\small{Statistics for the tilt-twist equilibrium solution with the different formulations and penalty weights. Included is the system free energy, the computed $L^2$-error on the finest grid, and the minimum and maximum deviations from unit director length at the quadrature nodes. Approximations of the cost in WUs for the corresponding method with no trust regions and simple trust regions are included.}}
\label{TiltTwistBCFormulationComparisons}
\end{table}

\begin{table}[h!]
\centering
{\small
\begin{tabular}{|c|c|c|c|c|c|c|}
\hline
& \multicolumn{2}{|c|}{No Trust Region} & \multicolumn{2}{|c|}{Simple Trust Region} & \multicolumn{2}{|c|}{2D Trust Region} \\
\hline
Type & $L^2$-error & Cost & $L^2$-error & Cost & $L^2$-error & Cost \\
\hline
Pen. $\zeta = 10^1$ & $4.354$e-01 & $1.384$ & $4.319$e-01  & $1.337$ & $4.424$e-01 & $1.377$ \\
\hline
Pen. $\zeta = 10^2$ & $3.691$e-02 & $1.335$ & $3.635$e-02 & $1.333$ & $3.578$e-02 & $1.333$ \\
\hline
Pen. $\zeta = 10^3$ & $4.708$e-03 & $1.335$ & $4.533$e-03 & $1.332$ & $4.493$e-03 & $1.332$ \\
\hline
Pen. $\zeta = 10^4$ & $1.085$e-03 & $1.335$ & $8.662$e-04 & $1.332$ & $8.536$e-04 & $1.333$ \\
\hline
Pen. $\zeta = 10^5$ & $1.650$e-03 & $1.336$ & $9.487$e-04 & $1.333$ & $7.012$e-04 & $1.333$ \\
\hline
Pen. $\zeta = 10^6$ & $9.414$e-01 & $87.362$ & $5.375$e-04 & $1.341$ & $7.344$e-04 & $1.341$ \\
\hline
\end{tabular}
}
\caption{\small{A comparison of renormalization penalty methods, with and without trust-region approaches, for the tilt-twist solution. For each algorithm, the computed $L^2$-error on the finest grid and an approximation of the cost in WUs is included.}}
\label{TiltTwistBCFormulationRenormPenaltyTRs}
\end{table}

Considering Tables \ref{TiltTwistBCFormulationRenormPenaltyTRs} and \ref{TiltTwistBCFormulationComparisonsRenormPenalty}, the renormalization penalty method does not perform as well as in the previous problem. Note that Table \ref{TiltTwistBCFormulationComparisonsRenormPenalty} details additional statistics for the 2D-subspace minimization trust-region approach discussed in Table \ref{TiltTwistBCFormulationRenormPenaltyTRs}. Compared to the unmodified penalty method, computational costs remain steadier, with the exception of the run without trust regions and a penalty weight of $10^6$, and adherence to the unit-length constraint is improved. However, the method fails to reach an equivalent accuracy before performance degrades. As with the simpler twist problem, performance of the renormalization method is sensitive to an appropriate choice of penalty weight. 

\begin{table}[h!]
\centering
{\small
\begin{tabular}{|c|c|c|c|c|c|}
\hline
Type & Free Energy & $L^2$-error & Min. Dev. & Max Dev. & 2D TR Cost \\
\hline
Pen. $\zeta = 10^1$ & $3.92827$ & $4.424$e-01 & $-4.62$e-09 & $4.62$e-09 & $1.377$ \\
\hline
Pen. $\zeta = 10^2$ & $3.59611$ & $3.578$e-02 & $-1.15$e-09 & $1.14$e-09 & $1.333$ \\
\hline
Pen. $\zeta = 10^3$ & $ 3.59298$ & $4.493$e-03 & $-7.77$e-10 & $7.76$e-10 & $1.332$  \\
\hline
Pen. $\zeta = 10^4$ & $3.59294$ & $8.536$e-04 & $-7.87$e-10 & $7.85$e-10 & $1.333$ \\
\hline
Pen. $\zeta = 10^5$ & $3.59294$ & $7.012$e-04 & $-7.91$e-10 & $7.90$e-10 & $1.333$ \\
\hline
Pen. $\zeta = 10^6$ & $3.59294$ & $7.344$e-04 & $-7.87$e-10 & $7.86$e-10 & $1.341$ \\
\hline
\end{tabular}
}
\caption{\small{Statistics for the tilt-twist equilibrium solution with varying penalty weights. Here, the penalty method with renormalization and 2D-subspace minimization is shown. Included is the system free energy, the computed $L^2$-error on the finest grid, the minimum and maximum deviations from unit director length at the quadrature nodes, and an approximation of the cost in WUs for the corresponding method.}}
\label{TiltTwistBCFormulationComparisonsRenormPenalty}
\end{table}

The method with renormalization does find the true free energy at a lower penalty weight than the approach without renormalization. At a penalty weight of $\zeta = 10^4$, the penalty method without renormalization has a slightly lower error measure, but has not accurately matched the true energy. While the unmodified method more accurately resolves the orientation of the director in comparison with the renormalization method, it slightly shrinks the director length to attain the moderately smaller free energy.

Figure \ref{fig:TiltNIFig} presents similar behavior to Figure \ref{fig:TwistNIFig}, in that trust regions productively reduce the number of iterations on the coarsest grids but have less effect on iteration counts on the finest grids. This is again due to the efficiency of nested iteration. Figure \ref{fig:TiltHelv} displays the solution computed by the Lagrange multiplier approach. For the unmodified penalty method with NI and a penalty weight of $\zeta = 10^9$, trust regions only reduce the computational cost from $3.218$ WUs to $2.879$ WUs. This results in only an $8.9\%$ decrease in overall time to solution.

\begin{figure}[h!]
\centering
\begin{subfigure}[b]{.49 \textwidth}
\raggedleft
  \includegraphics[scale=.32]{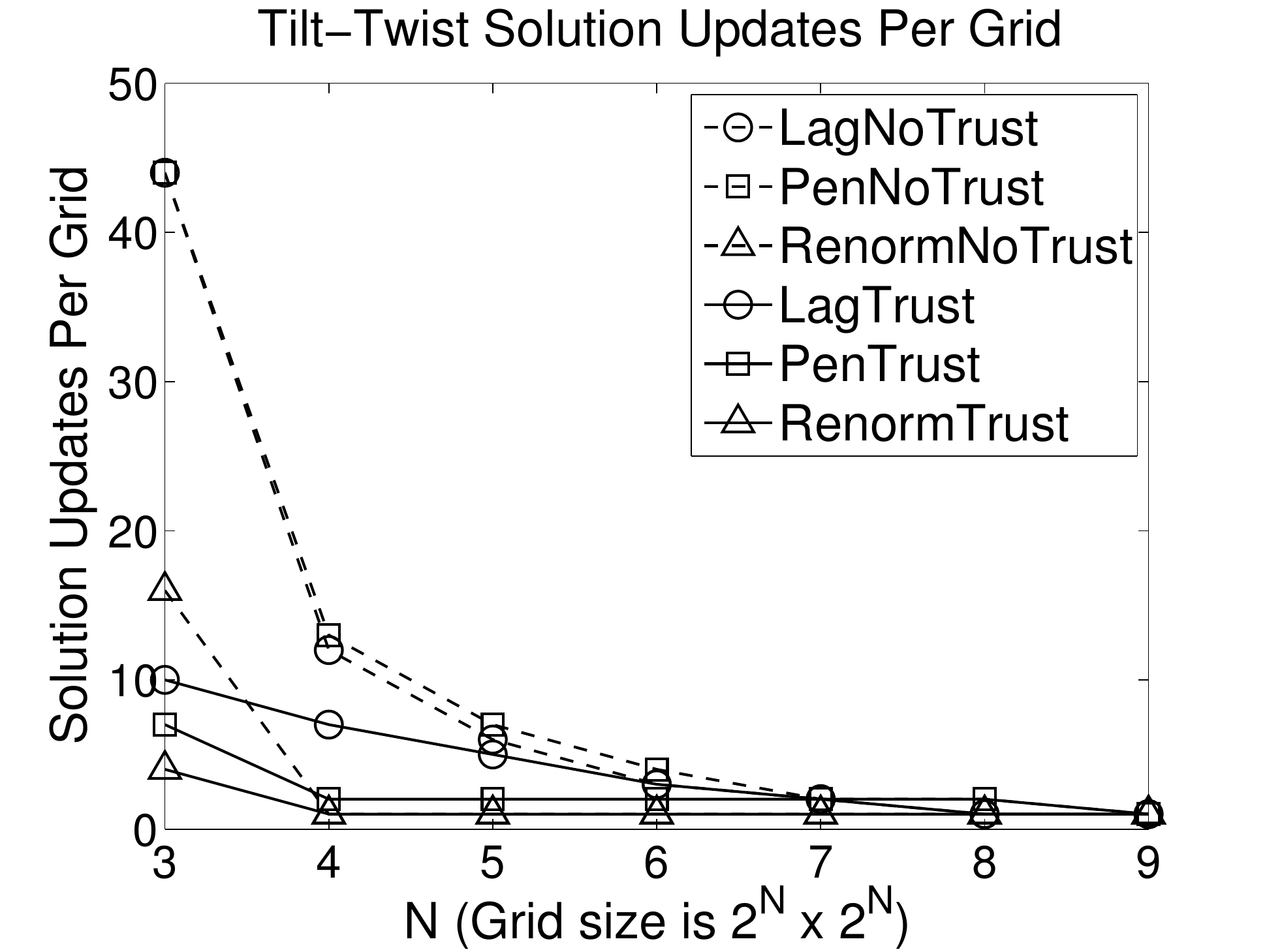}
  \caption{}
  \label{fig:TiltNIFig}
\end{subfigure}
\begin{subfigure}[b]{.49 \textwidth}
\raggedright
  \includegraphics[scale=.31]{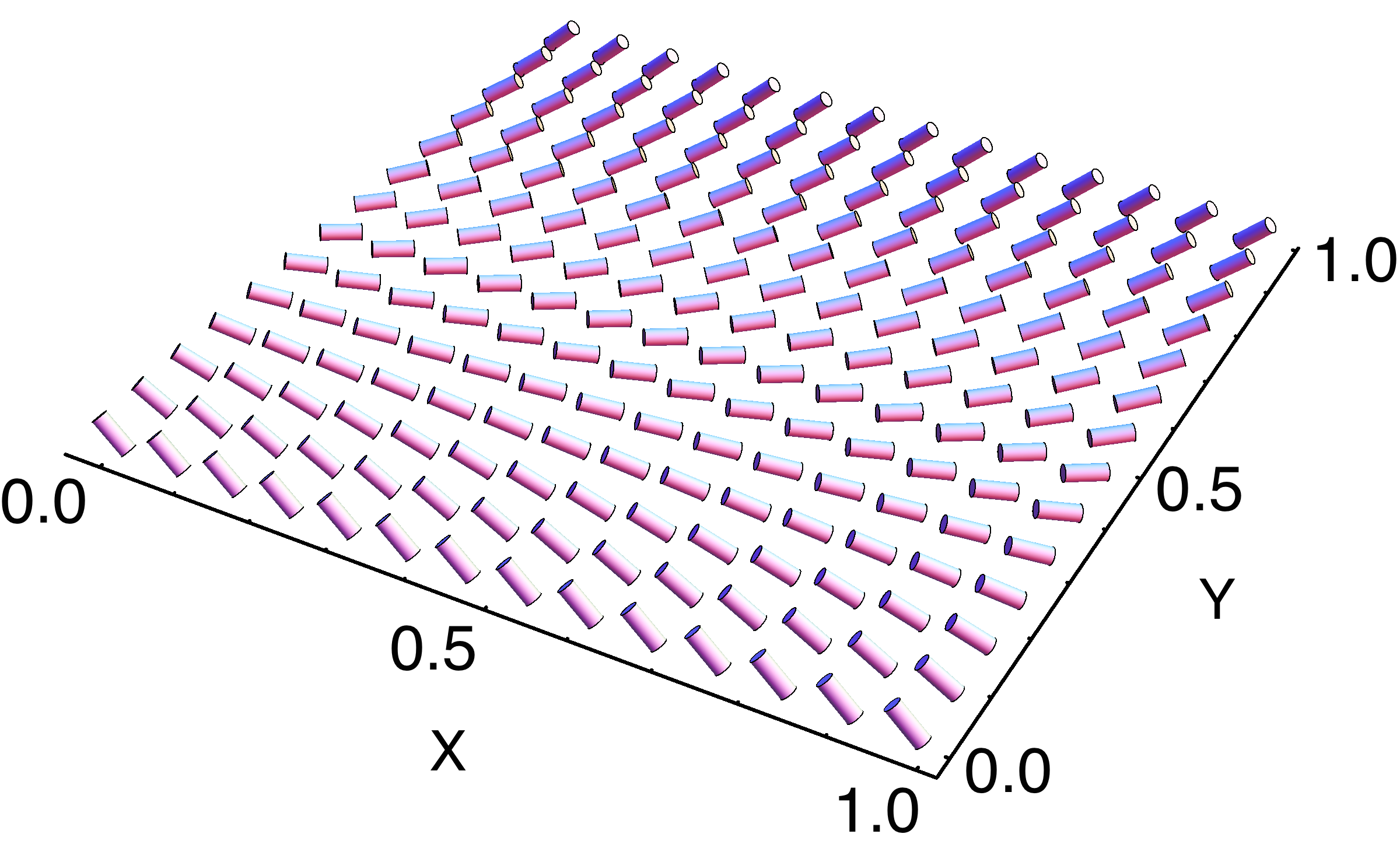}
  \caption{}
  \label{fig:TiltHelv}
\end{subfigure}
\caption{\small{(\subref{fig:TiltNIFig}) Number of iterations required to reach iteration tolerance for each method with NI. The penalty weight for the penalty formulation was $\zeta = 1000$. Only the 2D-subspace minimization trust-region approach is displayed, as the behavior of simple trust regions is similar. (\subref{fig:TiltHelv}) The final computed solution for the Lagrangian formulation on a $512 \times 512$ mesh (restricted for visualization).}}
\label{TiltNIFigure}
\end{figure} 

\begin{table}[h!]
{\small
\begin{minipage}[t]{.49 \linewidth}
\begin{tabular}[t]{|c|c|c|c|}
\hlinewd{1.3pt}
\multicolumn{4}{|c|}{Lagrangian}\\
\hlinewd{1.3pt}
\multicolumn{2}{|c|}{Method} & Solve Cost & Run Time\\
\hline
No NI & No TR & $33$ & $9,853$s \\
\hline
NI & No TR& $1.463$ & $584$s \\
\hline
No NI & TR & $9$ & $2,812$s  \\
\hline
NI & TR & $1.447$ & $579$s  \\
\hlinewd{1.3pt}
\multicolumn{4}{|c|}{Renormalization Penalty: $\zeta = 10^5$}\\
\hlinewd{1.3pt}
\multicolumn{2}{|c|}{Method} & Solve Cost & Run Time\\
\hline
No NI & No TR & $16$ & $5,119$s\\
\hline
NI & No TR & $1.336$ & $586$s\\
\hline
No NI & TR & $18$ & $5,658$s \\
\hline
No NI & TR 2D & $18$ & $5,817$s \\
\hline
NI & TR & $1.333$ & $575$s \\
\hline
NI & TR 2D & $1.333$ & $591$s \\
\hline
\end{tabular}
\end{minipage}
\begin{minipage}[t]{.49 \linewidth}
\begin{tabular}[t]{|c|c|c|c|}
\hlinewd{1.3pt}
\multicolumn{4}{|c|}{Unmodified Penalty $\zeta = 10^5$}\\
\hlinewd{1.3pt}
\multicolumn{2}{|c|}{Method} & Solve Cost & Run Time\\
\hline
No NI & No TR & $39$ & $11,606$s\\
\hline
NI & No TR & $2.743$ & $939$s\\
\hline
No NI & TR & $22$ & $6,598$s \\
\hline
No NI & TR 2D & $22$ & $6,680$s \\
\hline
NI &TR & $2.667$ & $920$s  \\
\hline
NI & TR 2D & $2.667$ & $949$s  \\
\hline
\end{tabular}
\end{minipage}
}
\caption{\small{Tilt-twist statistics comparison for NI and trust region combinations. The solve cost column displays an approximation of the work in WUs for the corresponding method. The overall time to solution is also presented.}}
\label{TiltTwistWUTable}
\end{table}

As shown in Table \ref{TiltTwistWUTable}, improvements from the trust regions for the Lagrange multiplier method are minor, decreasing computational costs from $1.463$ WUs to $1.447$ WUs and reducing overall time to solution by only $0.82\%$. Here, the renormalization penalty method is faster than the unmodified approach and, in some cases, even slightly outpaces the Lagrange multiplier formulation. However, as shown in Tables \ref{TiltTwistBCFormulationComparisons} and \ref{TiltTwistBCFormulationComparisonsRenormPenalty} the associated error convergence is not comparable. In Table \ref{TiltTwistWUTable}, results for $\zeta = 10^9$ are not reported due to untenably large run times without nested iteration. 

\subsection{Nano Patterned Boundary Conditions}

In this numerical experiment, we use Frank constants $K_1 = 1.0$, $K_2 = 0.62903$, and $K_3 = 1.32258$. The applied boundary conditions are the same as those used for the final experiment in \cite{Emerson1}. Letting $r = 0.25$ and $s = 0.95$, the boundary conditions are defined as
\begin{align}
n_1 &= 0, \label{nano1} \\
n_2 &= \cos\big(r(\pi + 2 \tan^{-1}(X_m) -2 \tan^{-1}(X_p))\big), \\
n_3 &= \sin\big(r(\pi + 2 \tan^{-1}(X_m) -2 \tan^{-1}(X_p))\big), \label{nano2}
\end{align}
where $X_m=\frac{-s\sin(2\pi(x+r))}{-s\cos(2\pi(x+r))-1}$ and $X_p = \frac{-s\sin(2\pi(x+r))}{-s\cos(2\pi(x+r))+1}$. These boundary conditions are pictured in Figure \ref{NanoNIFigure}\subref{fig:NanoHelv}. The result is a sharp transition from vertical nematics to planar-aligned rods followed by a rapid transition back to vertical alignment. Such boundary conditions produce configuration distortions throughout the interior of the domain and are important in physical applications \cite{Atherton1, Atherton2, Emerson2}. Due to this complexity, no analytical solution currently exists.

The more complicated nature of the nano-patterned boundary conditions is reflected in the data of Table \ref{NanoBCFormulationComparisons}. The overall approximate costs for the methods with and without trust regions are larger than previous examples and the unit-length constraint is more difficult to capture. Nonetheless, the Lagrange multiplier method provides an accurate and cost effective approach. The penalty method without trust regions diverges for penalty weights greater than $\zeta = 10^4$. At higher penalty weights, even the trust-region approach suffers jumps in computational costs. At $\zeta = 10^9$, the system becomes over constrained and accuracy begins to degrade. Hence, results for this weight are not included.

\begin{table}[h!]
\centering
{\small
\begin{tabular}{|c|c|c|c|c|c|}
\hline
Type & Free Energy & Min. Dev. & Max Dev. & Cost & TR Cost\\
\hlinewd{1.3pt}
Lagrangian & $3.89001$ & $-6.92$e-05 & $5.89$e-05 & $2.864$ & $2.779$\\
\hlinewd{1.3pt}
Pen. $\zeta = 10^1$ & $3.83657$ & $-8.84$e-02 & $1.96$e-03 & $2.864$ & $2.748$ \\
\hline
Pen. $\zeta = 10^2$ & $3.86896$ & $-4.01$e-02 & $4.40$e-03 & $2.864$ & $2.749$\\
\hline
Pen. $\zeta = 10^3$ & $3.88331$ & $-1.80$e-02 & $7.32$e-03 & $2.868$ & $2.749$ \\
\hline
Pen. $\zeta = 10^4$ & $3.88819$ & $-6.58$e-03 & $5.81$e-03 & $2.886$ & $2.757$\\
\hline
Pen. $\zeta = 10^5$ & $3.88965$  & $-1.60$e-03 & $2.01$e-03 & $-$ & $2.805$\\
\hline
Pen. $\zeta = 10^6$ & $3.88996$  & $-2.90$e-04 & $4.55$e-04 & $-$ &$3.736$\\
\hline
Pen. $\zeta = 10^7$ & $3.89001$  & $-7.92$e-05 & $1.01$e-04 & $-$ & $4.797$\\
\hline
Pen. $\zeta = 10^8$ & $3.89001$ & $-6.76$e-05 & $5.83$e-05 & $-$ & $22.328$\\
\hline
\end{tabular}
}
\caption{\small{Statistics for the nano-patterned equilibrium solution with the different formulations and penalty weights. Included is the system free energy and the minimum and maximum deviations from unit director length at the quadrature nodes. Approximations of the cost in WUs for the corresponding method with no trust regions and simple trust regions are included. Dashes in the columns indicate divergence.}}
\label{NanoBCFormulationComparisons}
\end{table}

As with the tilt-twist equilibrium solution, the renormalization penalty method approaches the Lagrangian formulation's free energy and unit-length constraint bounds earlier than the unmodified penalty method. It also yields a lower computational cost for most penalty weights. However, as was seen in the tilt-twist data, matching the energy earlier than the unmodified penalty approach does not directly indicate higher accuracy in resolving the correct orientation of the director. Moreover, in Table \ref{NanoBCFormulationRenormPenaltyTRs}, divergence issues are apparent for the renormalization method at high penalty weights. 

When considered with Table \ref{NanoBCFormulationComparisons}, Table \ref{NanoBCFormulationRenormPenaltyTRs} reinforces the conclusion that trust regions positively influence the robustness of penalty method approaches. While the simple trust-region approach works most effectively for the non-renormalization penalty method, the 2D-subspace minimization approach is more favorable for the renormalization penalty formulation. Note that Table \ref{NanoBCFormulationComparisonsRenormPenalty} details additional statistics for the 2D-subspace minimization trust-region approach discussed in Table \ref{NanoBCFormulationRenormPenaltyTRs}.

\begin{table}[h!]
\centering
{\small
\begin{tabular}{|c|c|c|c|c|c|c|}
\hline
& \multicolumn{2}{|c|}{No Trust Region} & \multicolumn{2}{|c|}{Simple Trust Region} & \multicolumn{2}{|c|}{2D Trust Region} \\
\hline
Type & Free Energy & Cost & Free Energy & Cost & Free Energy & Cost \\
\hline
Pen. $\zeta = 10^1$ & $3.89319$ & $1.680$ & $3.89351$  & $1.518$ & $3.89308$ & $1.686$ \\
\hline
Pen. $\zeta = 10^2$ & $3.89049$ & $1.681$ & $3.89051$ & $1.666$ & $3.89049$& $1.666$ \\
\hline
Pen. $\zeta = 10^3$ & $3.89006$ & $1.682$ & $3.89006$ & $1.666$ & $3.89006$& $1.666$ \\
\hline
Pen. $\zeta = 10^4$ & $3.89002$ & $1.683$ & $3.89002$ & $1.669$ & $3.89002$ & $1.669$ \\
\hline
Pen. $\zeta = 10^5$ & $-$ & $-$ & $3.89001$& $2.133$ & $3.89001$ & $2.433$ \\
\hline
Pen. $\zeta = 10^6$ & $-$ & $-$ & $-$ & $-$ & $3.89001$ & $5.418$ \\
\hline
\end{tabular}
}
\caption{\small{A comparison of renormalization penalty methods, with and without trust-region approaches, for the nano-pattern solution. For each algorithm, the computed free energy on the finest grid and an approximation of the cost in WUs is included. Dashes in the columns indicate divergence.}}
\label{NanoBCFormulationRenormPenaltyTRs}
\end{table}

\begin{table}[h!]
\centering
{\small
\begin{tabular}{|c|c|c|c|c|}
\hline
Type & Free Energy & Min. Dev. & Max Dev. & 2D TR Cost \\
\hline
Pen. $\zeta = 10^1$ & $3.89308$ & $-7.06$e-05 & $6.02$e-05 & $1.686$ \\
\hline
Pen. $\zeta = 10^2$ & $3.89049$ & $-7.07$e-05 & $6.02$e-05 & $1.666$ \\
\hline
Pen. $\zeta = 10^3$ & $ 3.89006$ & $-7.09$e-05 & $6.01$e-05 & $1.666$  \\
\hline
Pen. $\zeta = 10^4$ & $3.89002$ & $-7.09$e-05 & $6.01$e-05 & $1.669$ \\
\hline
Pen. $\zeta = 10^5$ & $3.89001$ & $-7.08$e-05 & $6.00$e-05 & $2.433$ \\
\hline
Pen. $\zeta = 10^6$ & $3.89001$ & $-7.07$e-05 & $5.98$e-05 & $5.418$ \\
\hline
\end{tabular}
}
\caption{\small{Statistics for the nano equilibrium solution with the different formulations and penalty weights. Here, the penalty method with renormalization and 2D-subspace minimization is used. Included is the system free energy, the minimum and maximum deviations from unit director length at the quadrature nodes, and an approximation of the cost in WUs for the corresponding method.}}
\label{NanoBCFormulationComparisonsRenormPenalty}
\end{table}

Figure \ref{fig:NanoNIFig} displays the iteration counts as a function of grid size for both the Lagrange multiplier formulation and the penalty method, with and without renormalization, at a penalty weight of $\zeta = 10^3$. The solution computed by the Lagrange multiplier method is shown in Figure \ref{fig:NanoHelv}. Similar to the previous problems, trust regions reduce iteration counts on the coarsest grids with reduced efficacy at the finer levels, due to NI. The cost savings from trust regions within a nested iteration scheme are slightly higher for this problem but persist as small improvements overall.

\begin{figure}[h!]
\centering
\begin{subfigure}[b]{.49 \textwidth}
\raggedleft
  \includegraphics[scale=.32]{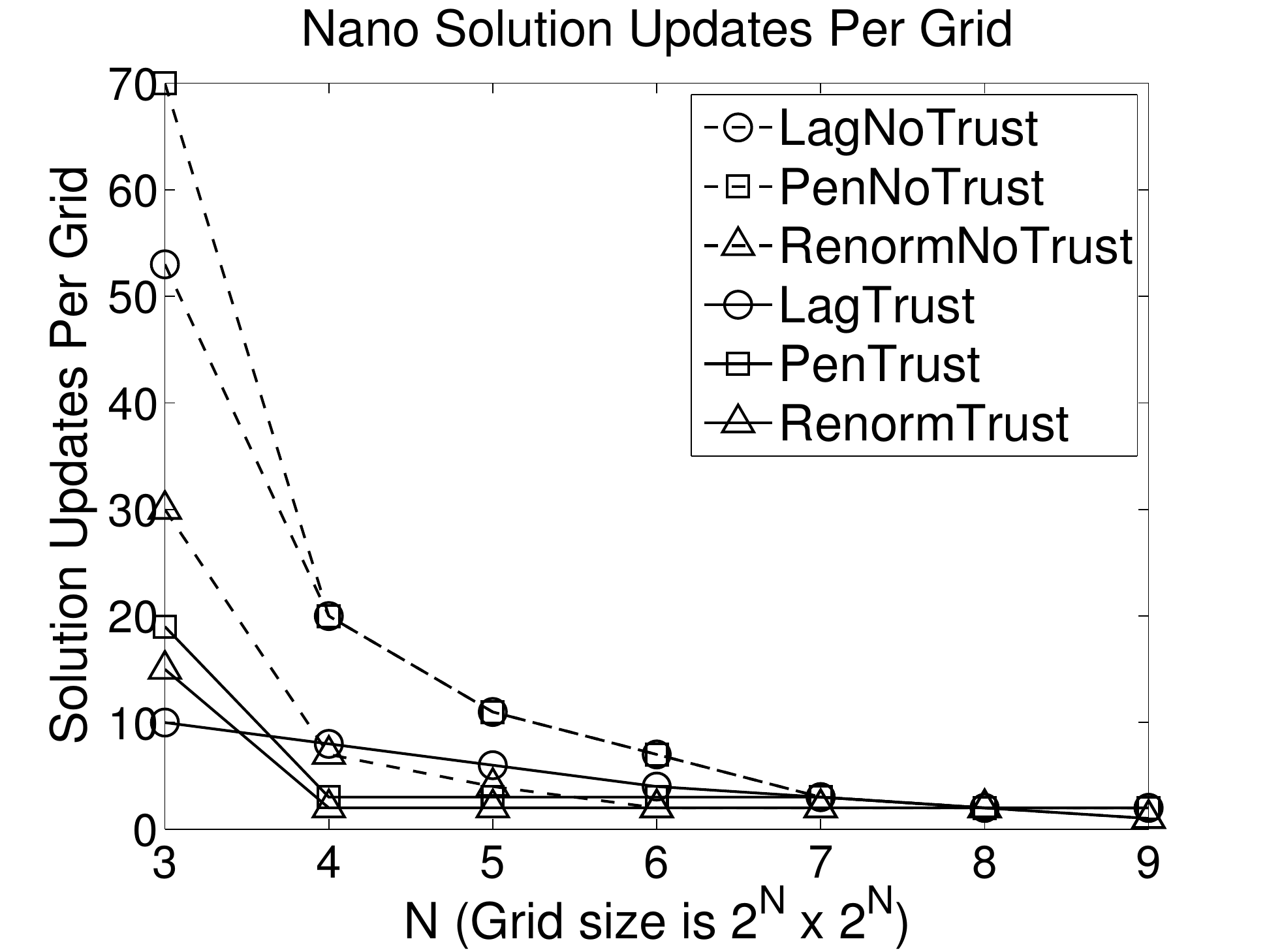}
  \caption{}
  \label{fig:NanoNIFig}
\end{subfigure}
\begin{subfigure}[b]{.49 \textwidth}
\raggedright
  \includegraphics[scale=.31]{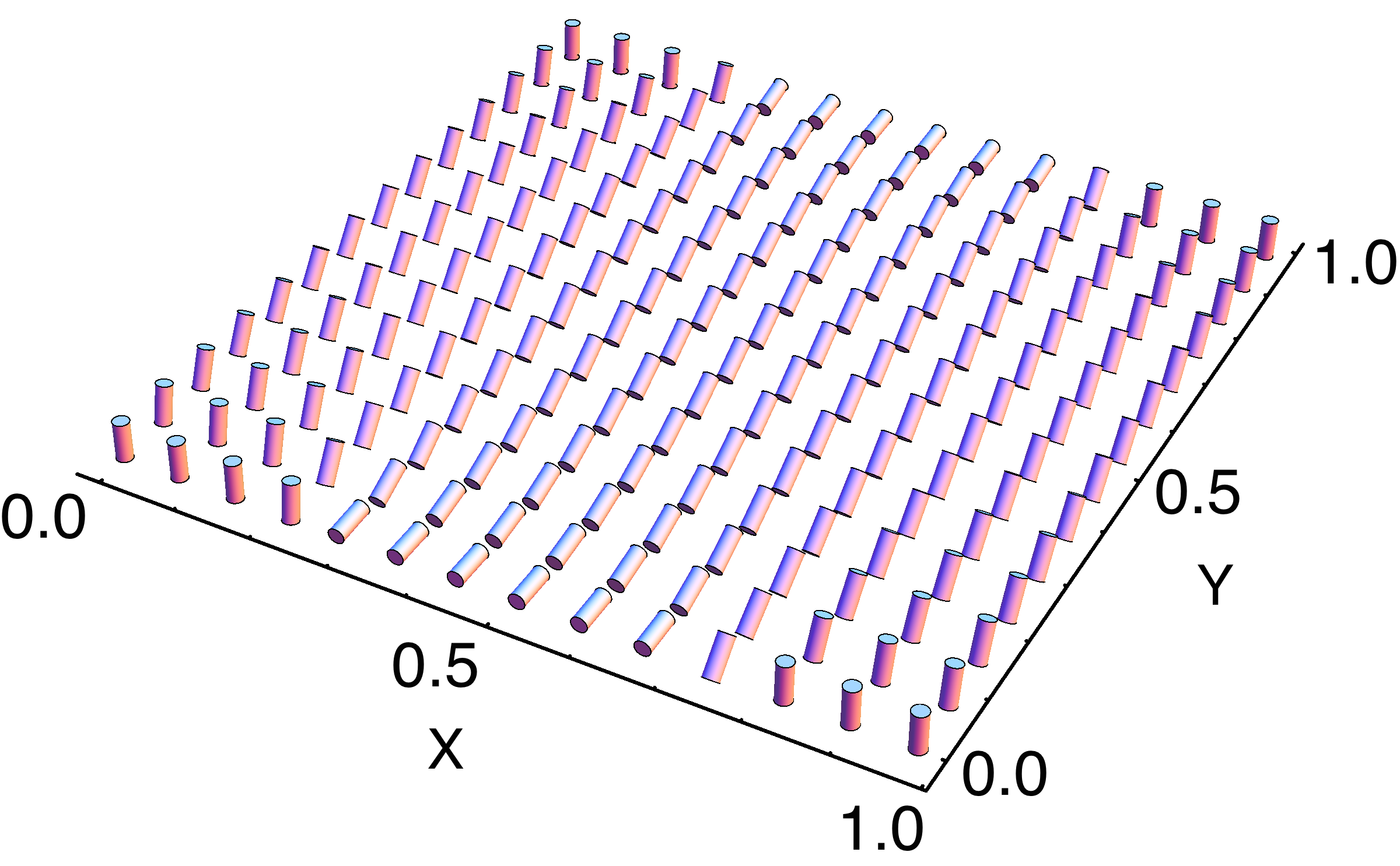}
  \caption{}
  \label{fig:NanoHelv}
\end{subfigure}
\caption{\small{(\subref{fig:NanoNIFig}) Number of iterations required to reach iteration tolerance for each method with NI. The penalty weight for the penalty formulation was $\zeta = 1000$. Only the 2D-subspace minimization trust-region approach is displayed, as the behavior of simple trust regions is similar. (\subref{fig:NanoHelv}) The final computed solution for the Lagrangian formulation on a $512 \times 512$ mesh (restricted for visualization).}}
\label{NanoNIFigure}
\end{figure}

Table \ref{NanoWUTable} reiterates the efficacy of nested iteration for efficient computation and trust regions for robustness. For this problem, Table \ref{NanoWUTable} also shows that the renormalization penalty method with nested iteration and trust regions has a somewhat shorter overall run time than that of the Lagrange multiplier approach. Moreover, the renormalization approach matches the free energy and unit-length conformance of the Lagrangian formulation, see Tables \ref{NanoBCFormulationComparisons} and \ref{NanoBCFormulationComparisonsRenormPenalty}. However, while the overall run time and approximate cost of the approach is slightly larger, the accuracy of the Lagrange multiplier formulation is expected to be much higher. For the Lagrange multiplier approach, the $l_2$-norm of the first-order optimality conditions is $7.386$e-$13$, whereas the same measure for the renormalization penalty method is $1.603$e-$02$. 

\begin{table}[h!]
{\small
\begin{minipage}[t]{.49 \linewidth}
\begin{tabular}[t]{|c|c|c|c|}
\hlinewd{1.3pt}
\multicolumn{4}{|c|}{Lagrangian}\\
\hlinewd{1.3pt}
\multicolumn{2}{|c|}{Method} & Solve Cost & Run Time\\
\hline
No NI & No TR & $63$ & $18,861$s \\
\hline
NI & No TR& $2.864$ & $983$s \\
\hline
No NI & TR & $10$ & $3,113$s  \\
\hline
NI & TR & $2.779$ & $960$s  \\
\hlinewd{1.3pt}
\multicolumn{4}{|c|}{Renormalization Penalty: $\zeta = 10^5$}\\
\hlinewd{1.3pt}
\multicolumn{2}{|c|}{Method} & Solve Cost & Run Time\\
\hline
No NI & No TR & $35$ & $10,918$s\\
\hline
NI & No TR & $-$ & $-$\\
\hline
No NI & TR & $32$ & $9,893$s \\
\hline
No NI & TR 2D & $34$ & $10,976$s \\
\hline
NI & TR & $2.133$ & $789$s \\
\hline
NI & TR 2D & $2.433$ & $901$s \\
\hline
\end{tabular}
\end{minipage}
\begin{minipage}[t]{.49 \linewidth}
\begin{tabular}[t]{|c|c|c|c|}
\hlinewd{1.3pt}
\multicolumn{4}{|c|}{Unmodified Penalty $\zeta = 10^5$}\\
\hlinewd{1.3pt}
\multicolumn{2}{|c|}{Method} & Solve Cost & Run Time\\
\hline
No NI & No TR & $169$ & $49,654$s\\
\hline
NI & No TR & $-$ & $-$\\
\hline
No NI & TR & $73$ & $21,415$s \\
\hline
No NI & TR 2D& $75$ & $22,366$s \\
\hline
NI & TR & $2.805$ & $958$s  \\
\hline
NI & TR 2D& $3.530$ & $1,202$s  \\
\hline
\end{tabular}
\end{minipage}
}
\caption{\small{Nano-pattern statistics comparison for NI and trust region combinations. The solve cost column displays an approximation of the work in WUs for the corresponding method. The overall time to solution is also presented. Dashes in the columns indicate divergence.}}
\label{NanoWUTable}
\end{table}

In all of the experiments above, the accuracy per unit cost of the Lagrange multiplier method convincingly outperforms that of either of the penalty methods. Moreover, the experimental results imply that nested iteration should be used when considering any of the methods. While trust regions offer very slight improvements in computation time, they readily improve robustness of the penalty method. Due to their limited cost, it would be advantageous to include them for either method. The simple trust-region approach works best for the unmodified penalty method with stopping tolerances based on the first-order optimality conditions, whereas the 2D-subspace minimization trust regions are most effective for the renormalization penalty method with an energy reduction based stopping tolerance. Though larger penalty weights are generally necessary, the unmodified penalty method offers more consistent error reduction and performance with respect to an increasing weight.

\section{Multigrid Solver} \label{BraessSarazinSmoothing}

With the superiority of the Lagrange multiplier approach paired with nested iteration and trust regions established above, we demonstrate the full efficiency of the approach on a problem with electric and flexoelectric coupling as well as more complicated nano-patterened boundary conditions. In addition, a highly efficient, coupled multigrid method for the associated linear systems is introduced. The boundary conditions considered are a doubling of the nano-pattern described by Equations \eqref{nano1} - \eqref{nano2}, such that the pattern contains a second strip parallel to the $xy$-plane. While there is no applied electric field, the curvature induced by the nano-patterning generates an internal electric field due to the flexoelectric properties of the liquid crystals \cite{Meyer1, Emerson2}. The flexoelectric free energy with appropriate boundary conditions is given as
\begin{align*}
\mathcal{F}(\director, \phi) &= K_1 \Ltwonorm{\diverg \director}{\Omega}^2 + K_3\Ltwoinnerndim{\vec{Z} \curl \director}{\curl \director}{\Omega}{3} \nonumber \\
& \qquad - \epsilon_0\epsilon_{\perp}\Ltwoinnerndim{\nabla \phi}{\nabla \phi}{\Omega}{3} - \epsilon_0 \epsilon_a \Ltwoinner{\director \cdot \nabla \phi}{\director \cdot \nabla \phi}{\Omega} \nonumber \\
& \qquad + 2e_s \Ltwoinner{\diverg \director}{\director \cdot \nabla \phi}{\Omega} + 2e_b\Ltwoinnerndim{\director \times \curl \director}{\nabla \phi}{\Omega}{3}.
\end{align*}
The Frank constants used are $K_1=K_3=1.0$ and $K_2=4.0$, and the electric constants are set to $\epsilon_{\parallel} = 7.0$, $\epsilon_{\perp} = 7.0$, $\epsilon_a = 0$, and $\epsilon_0 = 1.42809$. For the flexoelectric constants, using Rudquist notation \cite{Rudquist1}, $e_s = 1.5$ and $e_b = -1.5$.

For a full derivation of the Lagrange multiplier approach extended to include electric and flexoelectric effects, see \cite{Emerson2}. Using an electric potential, $\phi$, the discretized and linearized system is written
\begin{equation*}
\mathcal{M} 
\left [ \begin{array}{c}
\director \\
\phi \\
\lambda
\end{array} \right] =
\left [ \begin{array}{c c c}
A & B_1 & B_2 \\
B_1^T & -D & \vec{0} \\
B_2^T & \vec{0} & \vec{0}
\end{array} \right ]
\left [ \begin{array}{c}
\director \\
\phi \\
\lambda
\end{array} \right]  = 
\left [ \begin{array}{c}
f_{\director} \\
f_{\phi} \\
f_{\lambda}
\end{array} \right].
\end{equation*}
Define blocks of $\mathcal{M}$ as
\begin{align}
\hat{A} = \left [ \begin{array}{c c}
A & B_1 \\
B_1^T & -D
\end{array} \right ], & &
\hat{B} = \left [ \begin{array}{c}
B_2 \\ \vec{0}
\end{array} \right ].
\end{align}
Furthermore, let $\hat{u} = [\director \text{ } \phi]^T$ and $\hat{f}_{\hat{u}} = [f_{\director} \text{ } f_{\phi}]^T$. With these block definitions, Braess-Sarazin relaxation, originally formulated in \cite{Braess2} for Stokes flows, is used and takes the form
\begin{equation} \label{BraessUpdate}
\left [\begin{array}{c}
\hat{u}_{k+1} \\
\lambda_{k+1}
\end{array} \right ] = 
\left [ \begin{array}{c}
\hat{u}_k \\
\lambda_k
\end{array} \right ] + 
\left [ \begin{array}{c c}
\gamma_b R & \hat{B} \\
\hat{B}^T & \vec{0}
\end{array} \right ]^{-1} \left (
\left [\begin{array}{c}
\hat{f}_{\hat{u}} \\
f_{\lambda}
\end{array} \right ] - 
\left [ \begin{array}{c c}
\hat{A} & \hat{B} \\
\hat{B}^T & \vec{0}
\end{array} \right ]
\left[ \begin{array}{c}
\hat{u}_k \\
\lambda_k
\end{array} \right ]
\right ),
\end{equation}
where $R$ is an appropriate preconditioner for $\hat{A}$ and $\gamma_b$ is a weighting parameter. For the multigrid approach, the matrix,
\begin{equation*}
\left [ \begin{array}{c c}
\gamma_b R & \hat{B} \\
\hat{B}^T & \vec{0}
\end{array} \right ],
\end{equation*}
is only approximately inverted, as in \cite{Benson1}. We use $Q_2$ elements for both $\director$ and $\phi$, and, hence, the degrees of freedom for the components of $\director$ and $\phi$ are collocated. As suggested in \cite{Benson1}, we construct the approximation, $R$, by extracting $4 \times 4$-blocks of $\hat{A}$ corresponding to the nodally collocated degrees of freedom for $\director$ and $\phi$. With careful permutation of the degrees of freedom in Equation \eqref{BraessUpdate}, $R$ becomes a block-diagonal matrix consisting of these $4 \times 4$ collocation blocks. 

Comparative studies of Braess-Sarazin-type relaxation schemes in a multigrid framework for Stokes flows are found in \cite{Larin1}, while numerical studies of their extension to multigrid methods for magnetohydrodynamic equations are performed in \cite{Benson1}. Here, as part of the underlying multigrid method, we use standard finite-element interpolation operators and Galerkin coarsening.

We first focus on determining the optimal value of the relaxation parameter. Here, and in the subsequent runs, the multigrid convergence tolerance, which is based on a ratio of the current solution's residual to that of the initial guess, remains fixed at $10^{-6}$ for each grid level and Newton step. The parameter $\gamma_b$ was varied from $1.10$ to $2.00$ in increments of $0.05$. Displayed in Figure \ref{BraessParameterStudy}\subref{fig:left} are the multigrid iteration counts, averaged over Newton iterations, for a $512 \times 512$ grid with respect to varying values of $\gamma_b$. The parameter study suggests that a $\gamma_b$ value of $1.20$ is optimal for convergence and, thus, we use this value. It is interesting to note that the iteration counts are relatively insensitive to increases in $\gamma_b$ above $1.20$.

\begin{figure}[h!]
\centering
\begin{subfigure}[b]{.49 \textwidth}
\raggedleft
  \includegraphics[scale=.35]{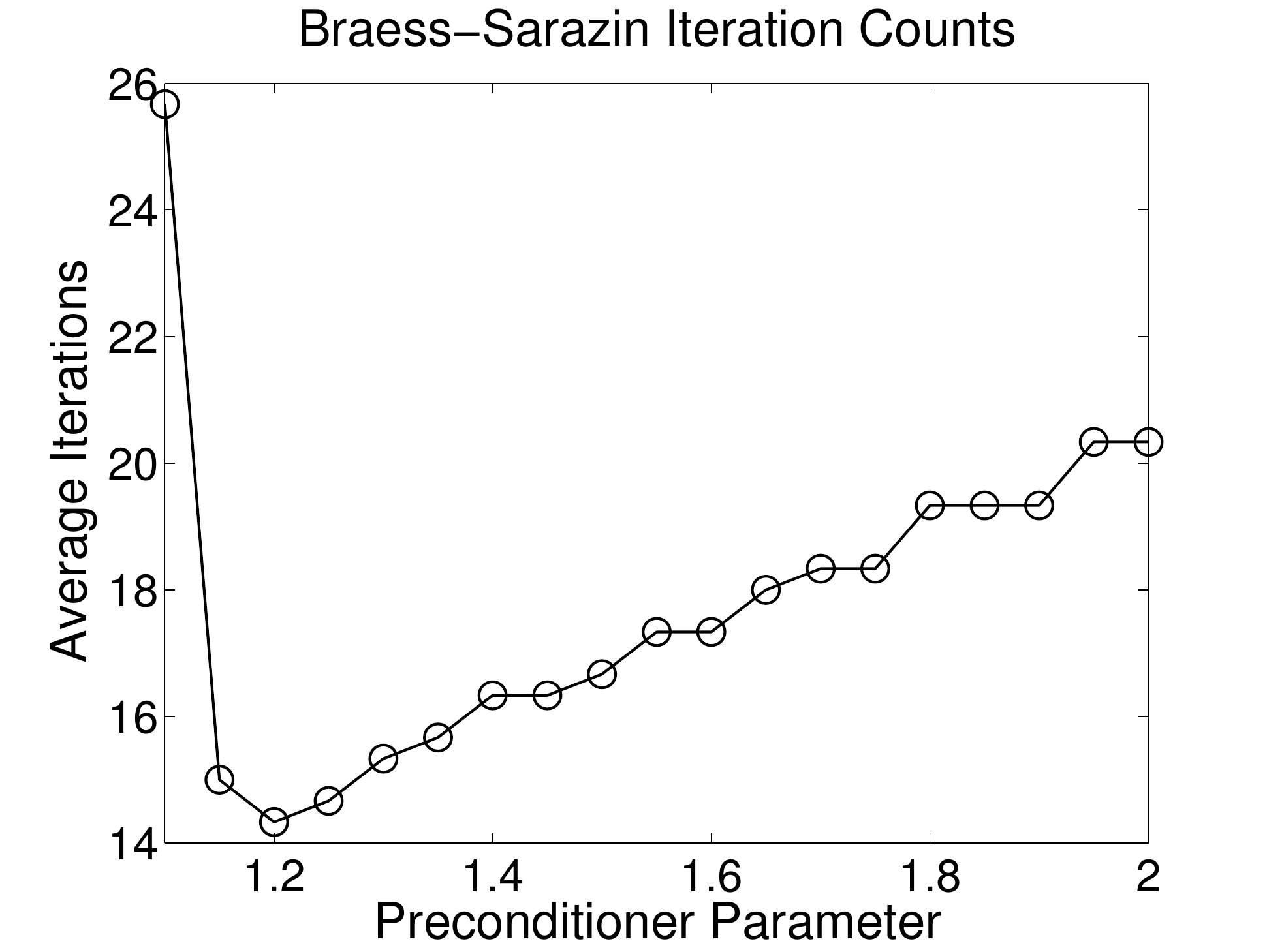}
  \caption{}
  \label{fig:left}
\end{subfigure}
\begin{subfigure}[b]{.49 \textwidth}
\raggedright
  \includegraphics[scale=.35]{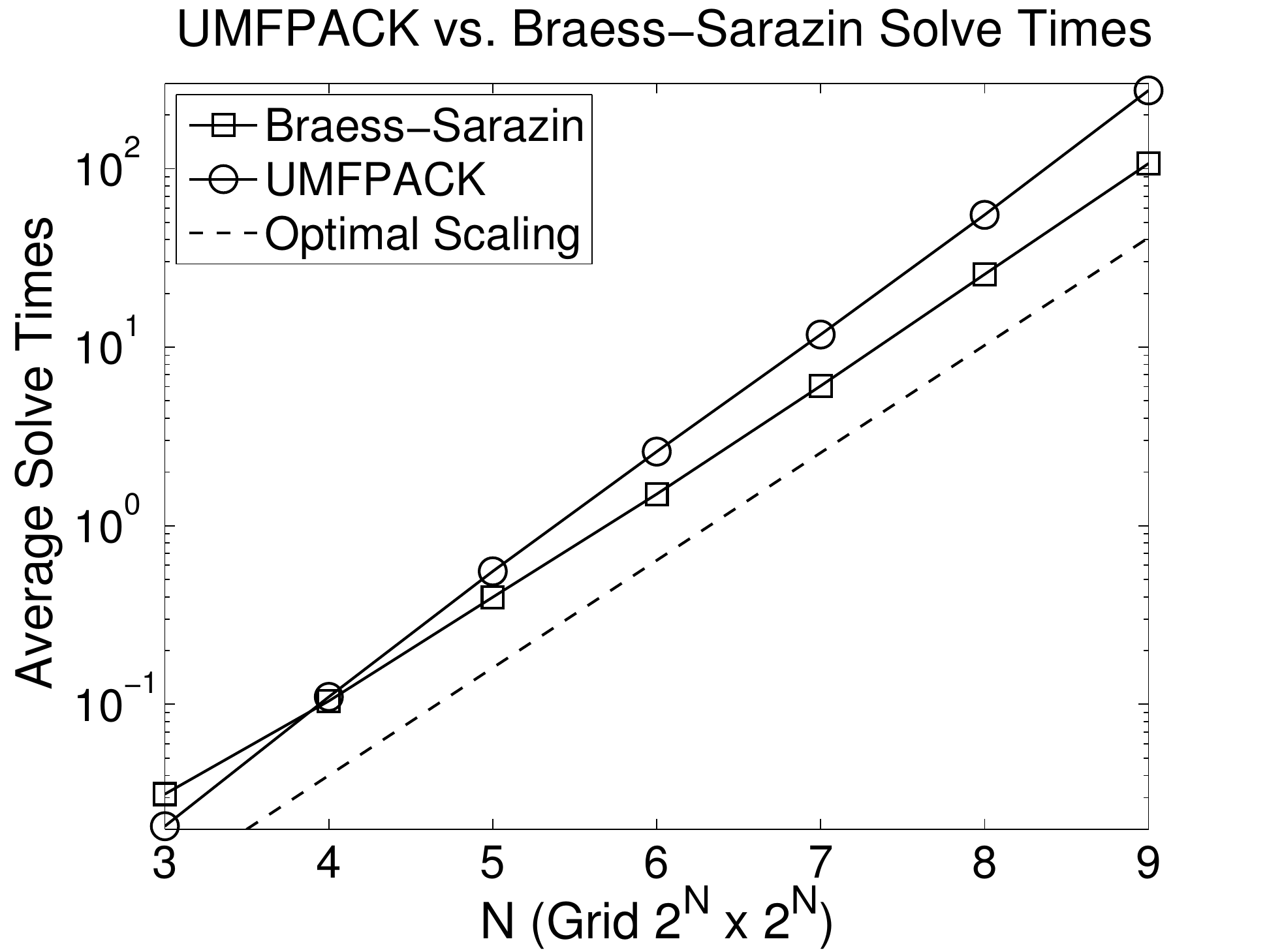}
  \caption{}
  \label{fig:right}
\end{subfigure}
\caption{\small{(\subref{fig:left}) The average number of multigrid iterations for varying $\gamma_b$ on a $512 \times 512$ grid. (\subref{fig:right}) The average time to solution for the Braess-Sarazin multigrid scheme compared to the UMFPACK direct solver.}}
\label{BraessParameterStudy}
\end{figure}

Figure \ref{BraessParameterStudy}\subref{fig:right} exhibits average total setup and solve times for grid sizes from $8 \times 8$ to $512 \times 512$ for the UMFPACK direct solver and the Braess-Sarazin-type multigrid scheme. Using $Q_2$--$Q_2$--$P_0$ elements for $\director$, $\phi$, and $\lambda$, respectively, the matrices on the $512 \times 512$ grid are of dimension $4,464,644 \times 4,464,644$ with $289,969,900$ nonzero entries. Here, the Braess-Sarazin multigrid scheme is scaling optimally with the grid size. Furthermore, there is a clear timing crossover at the $16 \times 16$ mesh, at which point Braess-Sarazin becomes the faster solver. This timing intersection occurs considerably earlier than the Vanka-type relaxation scheme discussed in \cite{Emerson2}.

\begin{table}[h!]
\centering
{\small
\begin{tabular}{|c|c|c|c|c|}
\hline
& \multicolumn{2}{|c|}{UMFPACK Solve} & \multicolumn{2}{|c|}{Braess-Sarazin Solve} \\
\hline
Trust-Region Type & None & Simple & None & Simple \\
\hlinewd{1.3pt}
System Assembly Time & $136.1$s & $131.3$s & $136.8$s & $131.8$s \\
\hline
Data Conversion Time & $-$ & $-$ & $136.9$s & $132.6$s \\
\hline
Linear Setup/Solve Time & $1053.2$s & $1035.2$s & $436.8$s & $425.2$s \\
\hline
Memory/Output Time & $284.7$s & $283.2$s & $305.8$s & $302.6$s \\
\hlinewd{1.3pt}
Total Time & $1474.0$s & $1449.7$s & $1016.3$s & $992.2$s \\
\hline
\end{tabular}
}
\caption{\small{A comparison of timing breakdowns for runs using the UMFPACK direct solver or the Braess-Sarazin multigrid scheme. Each solver is run with and without trust regions.}}
\label{UMFPACKBraessStatistics}
\end{table}

Table \ref{UMFPACKBraessStatistics} details a comparison of the UMFPACK direct solver's performance to that of the Braess-Sarazin-type multigrid scheme. With and without trust regions, the computed free energy between the two solvers is identically $16.413$. However, Braess-Sarazin-type relaxation reduces overall runtime by nearly $33\%$. This speed up is most notable when considering the fact that overall runtime for the multigrid solver experiments includes porting variables to types compatible with the Trilinos computational library \cite{Trilinos1} and computing collocation information for the Braess-Sarazin-type relaxation. Using the Lagrange multiplier formulation, nested iteration, multigrid, and trust regions, we obtain a robust and efficient algorithm.

\section{Summary and Future Work} \label{conclusion}

We have discussed three approaches for imposing the pointwise unit-length constraint required by the Frank-Oseen free-energy model for nematic liquid crystals, including a Lagrange multiplier formulation and a penalty method, with and without renormalization. Theory establishing the well-posedness of the intermediate discrete linearization systems for the penalty method was presented. Such theory parallels the theory establishing the well-posedness of the Lagrange multiplier linearization systems proved in \cite{Emerson1}. In addition, trust-region methods specifically tailored to each of the formulations were discussed and implemented.

The ensuing algorithms were compared for three benchmark equilibrium problems. The experiments suggested that the Lagrange multiplier method coupled with nested iteration is the most accurate and efficient approach for enforcing the unit-length constraint. Trust-region schemes were shown to increase overall robustness with very little extra cost and, therefore, should be considered when using either constraint formulation. Furthermore, nested iteration was observed to be exceedingly effective at reducing computational costs for all problems and methods. Finally, a Braess-Sarazin-type multigrid method, based on work in \cite{Benson1}, was introduced for the Lagrange multiplier formulation and successfully applied to a highly difficult liquid crystal problem with electric and flexoelectric coupling. The method quickly and accurately computed the expected equilibrium configuration.

The current implementation utilizes uniform grid refinement to build the nested iteration hierarchy of grids. Future work will include study of adaptive refinement techniques. Because the energy minimization formulation does not yield an obvious a posteriori error estimator, new techniques will be explored to flag cells for refinement. With an accurate and efficient approach for liquid crystal equilibrium configurations established, research into the application of energy-minimization finite-element approaches to liquid crystal flow problems will also be undertaken.


\section*{Acknowledgments}
The authors would like to thank Professor Timothy Atherton for his useful suggestions and contributions. We would also like to thank Thomas Benson for allowing us to use and adapt his code, as well as his helpful contributions.


\bibliographystyle{plain}	


\bibliography{LiquidCrystalPenalty}		

\end{document}